\documentclass[12pt]{article}
\usepackage[english]{babel}
\usepackage[cp1251]{inputenc}
\usepackage{amsmath,amssymb,graphicx,amsthm,enumerate,comment,amscd}
\usepackage[unicode,pdftex]{hyperref}
\usepackage[matrix,arrow,curve]{xy}

\newtheorem{theorem}{Theorem}
\newtheorem*{theoremwn}{Theorem}
\newtheorem{corollary}{Corollary}
\newtheorem{lemma}{Lemma}
\newtheorem{proposition}{Proposition}
\newtheorem*{question}{Question}

\theoremstyle{remark}
\newtheorem{remark}{Remark}

\theoremstyle{definition}
\newtheorem{definition}{Definition}

\newcommand{\R}{\mathbb{R}}

\newcommand{\Z}{\mathbb{Z}}

\newcommand{\del}{\smash{\mskip3mu\lower1truept\hbox{\vdots}\mskip3mu}}
\newcommand{\es}{\varnothing}
\renewcommand{\phi}{\varphi}
\newcommand{\eps}{\varepsilon}
\renewcommand{\kappa}{\varkappa}
\newcommand{\Id}{\mathrm{Id}}
\renewcommand{\Im}{\mathrm{Im}}
\newcommand{\tkr}{\mathrm{tkr}}

\newcommand{\sk}{\mathrm{sk}}
\renewcommand{\int}{\mathrm{int}}
\newcommand{\Diff}{\mathrm{Diff}}

\newcommand{\Sing}{\mathrm{Sing}}

\newcommand{\m}{\mathfrak{m}}

\newcommand{\RP}{\mathbb{RP}}

\newcommand{\Op}{{{\mathcal{O}p}}}

\begin{document}

\author{Andrey Ryabichev%
\footnote{\ IUM; Moscow, Russia //
\texttt{ryabichev@179.ru} //
Supported in part by Young Russian Mathematics award}}
\title{Maps of manifolds of the same dimension with prescribed Thom-Boardman singularities}
\date{}


\maketitle

\begin{abstract}
In this paper we extend Y.\,Eliashberg's $h$-principle to arbitrary generic smooth maps of smooth manifolds.
Namely, we prove a necessary and sufficient condition for a continuous map of smooth manifolds of the same dimension to be homotopic to a generic map with a prescribed Thom-Boardman singularity $\Sigma^I$ at each point. 
In dimension 3 we rephrase these conditions in terms of the Stiefel-Whitney classes and the cohomology classes of the given loci of folds, cusps and swallowtail points.
\end{abstract}

\section{Introduction}\label{s:introduction}

All manifolds and maps between them are assumed to be infinitely smooth unless we explicitly specify otherwise.
All manifolds are assumed to be without boundary.
We fix two $n$-manifolds $M$ and $N$, $n>1$.

\subsection{Thom-Boardman maps and generic maps}\label{s:thom-boardman-generic}

We begin with the Thom-Boardman classification of singularities.
Let $f:M\to N$ be a smooth map.
For a nonnegative integer $i_1\le n$ denote the set of points $x\in M$ such that $\dim(\ker df(x))=i_1$ by $\Sigma^{i_1}(f)$.
For ``almost every'' map $f$ the set $\Sigma^{i_1}(f)\subset M$ is a
(locally closed)
smooth submanifold. 
Then for a nonnegative integer $i_2\le i_1$ let $\Sigma^{i_1,i_2}(f)$ be the set of points $x\in \Sigma^{i_1}(f)$
such that $\dim(\ker df|_{\Sigma^{i_1}(f)}(x))=i_2$.
Similarly, for every sequence $I=i_1,i_2,\ldots,i_r$ of integers such that $n\ge i_1\ge i_2\ge\ldots\ge i_r\ge0$
we have a subset $\Sigma^I(f)\subset M$.

Also, the set $\Sigma^I(f)$ can be defined as the preimage
of a certain submanifold $\Sigma^I(M,N)\subset J^r(M,N)$ of the $r$-jet space.
If the jet extension $j^r(f)$ is transversal to $\Sigma^I(M,N)$, then both definitions coincide
(the details can be found in \cite{boardman}, see also \cite[\S2]{avgz}).
For dimension reasons, we can take $r=n+1$.
In the sequel as $I$ we will take sequences $i_1,i_2,\ldots,i_r$ such that $n\ge i_1\ge i_2\ge\ldots\ge i_r\ge0$ and $i_1>0$.

The Thom-Boardman submanifolds $\Sigma^I(M,N)\subset J^r(M,N)$ for different $I$ do not intersect.
Let $\Sigma(M,N)$ be the set of {\it critical $r$-jets} of maps $M\to N$,
i.\,e.\ jets which correspond to the germs of maps at critical points.
It is well known that the Thom-Boardman decomposition
$\Sigma(M,N)=\bigsqcup_I\Sigma^I(M,N)$ is not a stratification \cite[p.\,48]{boardman}.
However, since the intersections of $\Sigma^I(M,N)\subset J^r(M,N)$
with the fibers of the projection $J^r(M,N)\to M\times N$ are (locally closed) algebraic subsets, 
one can choose a stratification of $\Sigma(M,N)$ which is a subdivision of the Thom-Boardman decomposition.
We review the definitions and discuss the existence of such a subdivision in~\S\ref{s:stratifications}.

If the jet extension $j^r(f):M\to J^r(M,N)$ of a map $f:M\to N$
is transversal to all $\Sigma^I(M,N)$, then we call $f$ a {\it Thom-Boardman map},
and if $j^r(f)$ is transversal to all strata of $\Sigma(M,N)$, then we say that $f$ is {\it generic}.
Clearly, every generic map is Thom-Boardman.

By the Thom transversality theorem
generic maps $M\to N$ form a dense open subset in the set of all $C^\infty$-maps $M\to N$ equipped with the strong Whitney topology
(see e.\,g.\ \cite[\S3.2]{hirsch-book}).
The set of Thom-Boardman maps is always dense in this space,
but it is not open for $n>3$ (see e.\,g.\ \cite{wilson}).

For a generic map $f$ the set of its critical points $\Sigma(f)=j^r(f)^{-1}(\Sigma(M,N))$
is a stratified subset of $M$.
Every stratum $C\subset\Sigma(f)$ belongs to a certain $\Sigma^I(f)$,
and for dimension reasons the sequence $I$ has a zero at the end.
Therefore the restriction $f|_C$ is an immersion $C\to N$.

\subsection{Prescribed singularities}\label{s:intro-germs}

We will now define the notion of ``prescribed singularities''.
Recall that if $n$ is sufficiently large,
then the Thom-Boardman singularities $\Sigma^I$
are in general not stable, see e.\,g.~\cite[\S3.7]{avgz}.
In particular, for a fixed $I$ the germs of a given generic map $f$ at two points $x,x'\in\Sigma^I(f)$ may not be equivalent.
Therefore we need to specify the 
germ of a generic map at each critical point.

Namely, we fix a nonempty closed subset $S\subset M$.
Suppose we are given a collection of open subsets $U_i\subset M$ such that $S\subset\bigcup U_i$,
and a collection of $n$-manifolds $V_i$ (here $i=1,2,\ldots$ runs through a fixed, possibly countably infinite, set of values).
Suppose for all $i$ we are given generic maps $\phi_i:U_i\to V_i$ such that $S=\bigcup\Sigma(\phi_i)$.

We call the collection $\{\phi_i\}$ {\it locally compatible},
if for every $i,j$ the germs of $\phi_i$ and $\phi_j$ at every point $x\in U_i\cap U_j$
are {\it $L$-equivalent}.
This means that there are neighborhoods
$U_x\ni x$, $V_x\supset\phi_i(U_x)$ and $V'_x\supset\phi_j(U_x)$
and a diffeomorphism $\beta_x:V_x\to V'_x$ such that $\beta_x\circ \phi_i|_{U_x}=\phi_j|_{U_x}$.
Note that, if the collection $\{\phi_i\}$ is locally compatible, then $S$ turns out to be a stratified subset
since $\phi_i$ are generic (see \S\ref{s:generic}, Theorem~\ref{th:sigma-stratification}).

We say that a map $f':M\to N$ has {\it prescribed singularities}
if $\Sigma(f')=S$ and for every $i$ the germs of $f'$ and $\phi_i$ at every point $x\in U_i$ are $L$-equivalent.
In this paper we address the following question.
\begin{question}\label{question}
When does there exist a generic map $M\to N$ that has prescribed singularities?
\end{question}

\subsection{Twisted tangent bundle}\label{s:intro-ttb}

Take a locally compatible collection of maps $\{\phi_i\}$ as above.
Next we will construct a rank $n$ vector bundle $T^{\{\phi_i\}}M$ over $M$, called the {\it twisted tangent bundle}.
It plays the main role in our construction.

First we set $U_0=V_0=M\setminus S$ and $\phi_0=\Id:U_0\to V_0$.
If we add $\phi_0$ to the collection $\{\phi_1,\phi_2,\ldots\}$, the collection remains locally compatible.
Note that then the domains of the maps cover the whole of $M$.

Next, for every $i=0,1,\ldots$ we define a vector bundle $E_i\to U_i$ as the pullback $\phi_i^*(TV_i)$.
Since the collection $\{\phi_i\}$ is locally compatible,
for every $i,j$ and every point $x\in U_i\cap U_j$
we have an isomorphism of the fibers $d\beta_x:E_i|_x\to E_j|_x$,
where the map $\beta_x$ establishes an $L$-equivalence of germs of $\phi_i$ and $\phi_j$ at $x$, see above.
The isomorphism $d\beta_x$ does not depend on the choice of~$\beta_x$ (see Proposition~\ref{pr:gluing-wcwm-bundle} in \S\ref{s:ttb}),
so we can canonically define an isomorphism $\Phi_{i,j}:E_i|_{U_i\cap U_j}\to E_j|_{U_i\cap U_j}$ over the whole of $U_i\cap U_j$.

The vector bundle $T^{\{\phi_i\}}M$ is obtained by gluing $E_i$ via $\Phi_{i,j}$, $i,j=0,1,\ldots$.
We discuss this definition in detail in \S\ref{s:ttb}.

We suggest the bundle $T^{\{\phi_i\}}M$ up to isomorphism can be described in terms of
the stratification of $S$ and some additional discrete data defined by the germs $\phi_i$,
but this conjecture remains unproven at present.

\subsection{The $h$-principle}

In \cite{eliashberg} Y.\,Eliashberg proved a theorem, which is important for global singularity theory.
It may be rephrased with our terminology in the following way.
We discuss its relative version in \S\ref{s:eliashbergs-theorem}. 

\begin{theoremwn}[{\cite[Th.~A]{eliashberg}}]
Suppose $S\subset M$ is a codimension $1$ submanifold and the maps $\{\phi_i\}$ have only fold singularities $\Sigma^{1,0}$.
Then a continuous map $f:M\to N$ is homotopic to a map $f'$ with folds in $S$ and with no other critical points
if and only if the vector bundles $T^{\{\phi_i\}}M$ and $f^*(TN)$ are isomorphic.
\end{theoremwn}

We generalize this theorem and allow the maps $\{\phi_i\}$ to have arbitrary generic singularities.
These maps need to be locally compatible (as in \S\ref{s:intro-germs})
for the construction to make sense.
The following is the main result of the paper.

\begin{theorem}\label{th:h-principle}
A continuous map $f:M\to N$ is homotopic to a generic map $f'$ with prescribed singularities
if and only if the vector bundles $T^{\{\phi_i\}}M$ and $f^*(TN)$ are isomorphic.
\end{theorem}

For our proof is necessary that the maps $\phi_i$ should be generic.
Namely, we use the fact that $S$ is a stratified subset
and for every stratum $C\subset S$ and every $i$
the restriction $\phi_i|_{C\cap U_i}$ is an immersion.
We study the properties of such maps in \S\ref{s:local}.
Also, our proof uses a relative version of Eliashberg's theorem (see~\S\ref{s:eliashbergs-theorem}),
so it is necessary that the subset of points of $S$ where $\phi_i$ have fold singularities should be nonempty.

\vspace{.5em}

\subsubsection*{Another approach to an $h$-principle}

There are important results of Y.\,Ando on maps with prescribed singularities.
In \cite{ando} he proved a criterion for a map between manifolds to be homotopic
to a map which has Thom-Boardman singularities {\it not worse} (lexicographically) then a given sequence $I$.
This criterion is stated in terms of jets.

The difference between Theorem~\ref{th:h-principle} and the Y.\,Ando's theorem is that
we start with a more detailed set of data (compatible germs at all points of the critical locus instead of a jet section over the whole of $M$)
and as a result we obtain a map with a prescribed singularity {\it at each point},
while with Ando's approach one does not control
either the singularities that appear
or the loci of singularities of each type.

\subsection{Application to $3$-manifolds}

Suppose $n=3$.
Then in order to specify prescribed singularities of a generic map
one does not need to specify germ at each point of $S$.
Namely, for a generic map $f:M\to N$ the set of critical points $\Sigma(f)$
is the union of
the fold points $\Sigma^{1,0}(f)$,
the cusp points $\Sigma^{1,1,0}(f)$ and
the swallowtail points $\Sigma^{1,1,1,0}(f)$.
These are called {\it Morin singularities}.
These singularities are stable
and singularities of the same type are always equivalent \cite{morin}, \cite{whitney-plane}.

Recall that for every $I$ the set $\Sigma^I\subset J^r(M,N)$ is a smooth submanifold (\cite[\S6]{boardman}).
So for a generic map $f:M\to N$ of closed 3-manifolds
the set $\Sigma^1(f)\subset M$ is a smooth closed 2-submanifold,
the set $\Sigma^{1,1}(f)\subset\Sigma^{1}(f)$ is a smooth closed 1-submanifold and
the set $\Sigma^{1,1,1}(f)=\Sigma^{1,1,1,0}(f)\subset\Sigma^{1,1}(f)$ is a discrete subset.
See the details e.\,g.\ in \cite[pp.\,47-48]{avgz}.

We deduce from Theorem~\ref{th:h-principle} the following necessary and sufficient condition for a given continuous map of 3-manifolds
to be homotopic to a generic one that has a prescribed singular locus.
In the sequel, for a codimension $k$ submanifold $M'\subset M$ we denote
the Poincar\'e dual class of $M'$ in $H^k(M;\Z_2)$ by $[M']$.

\begin{theorem}\label{th:three-manifolds}
Suppose $M,N$ are closed $3$-manifolds, $S\subset M$ is a nonempty closed $2$-submanifold,
$C\subset S$ is a closed $1$-submanifold and $P\subset C$ is a discrete subset.
A continuous map $f:M\to N$ is homotopic to a generic map $f'$ such that
$\Sigma^{1}(f')=S$, $\Sigma^{1,1}(f')=C$ and $\Sigma^{1,1,1}(f')=P$
if and only if
\begin{itemize}
\item $[S]=w_1(M)+f^*w_1(N)$,
\item $[C]=w_1(M)\cdot[S]$ and
\item for every component $C'\subset C$ we have $[C']\cdot[S]\equiv |P\cap C'|\mod2$.
\end{itemize}
\end{theorem}

The last condition is needed in order for a characteristic vector field on $C\setminus P$
(see \S\ref{s:3-germs}, Fig.~\ref{fig:characteristic-field-in-cusp})
to exist.
Moreover, if this condition holds, there are exactly two such vector fields,
and a generic map $f'$ can be constructed for each of them.


Note that if a map $f'$ as in Theorem~\ref{th:three-manifolds} exists,
then the conditions of Theorem~\ref{th:three-manifolds}
agree with the Thom polynomial of the Morin singularities.
Namely, denote the degree $i$ component of $f'^*w(N)/w(M)$ by $\bar w_i$.
Then it is well known that
$[\Sigma^1(f')]=\bar w_1$,
$[\Sigma^{1,1}(f')]=\bar w_1^2+\bar w_2$ and
$[\Sigma^{1,1,1}(f')]=\bar w_1^3+\bar w_1\bar w_2$
(see e.\,g.\ \cite[p.\,191]{avgl}).
All these equalities follow from Theorem~\ref{th:three-manifolds}.

\subsubsection*{Other low-dimensional consequences}

In \cite{ryabichev-surfaces} we proved a necessary and sufficient condition for a map of surfaces
to be homotopic to a generic map with prescribed folds and cusps.
This result is similar to Theorem~\ref{th:three-manifolds}, but it also uses cohomology with local coefficients.
In \cite{ryabichev-surfaces} we also proved a necessary and sufficient condition for a pair of surfaces
to have a generic map between them which has prescribed folds and cusps.

In other words, for surfaces we can find a generic map with prescribed singularities {\it among all maps},
while for 3-manifolds we can solve a similar problem {\it only for maps from a given homotopy class}.
However, 
we can obtain e.\,g.\ the following results
on generic maps with prescribed singularities
using Theorem~\ref{th:three-manifolds}:

\begin{corollary}
Suppose $M,N$ are integral homology $3$-spheres.
Then every map $M\to N$ is homo\-topic to a 
map
with any given nonempty set of folds and with no other critical points.
\end{corollary}

\begin{corollary}
Suppose $M,N$ are closed orientable $3$-manifolds, $S\subset M$ is a nonempty codimension $1$ submanifold and $[S]=0$.
Then every map $M\to N$ is homotopic to a 
map with folds in $S$ and with no other critical points.
\end{corollary}

Presumably there are analogues of Theorem~\ref{th:three-manifolds} for manifolds of dimension 4 and higher.
In order to deduce them one would need to describe the twisted tangent bundle (and its characteristic classes)
in terms of strata of the singular locus.

\subsection*{Structure of the paper}
\addcontentsline{toc}{subsection}{Structure of the paper}

In \S\ref{s:stratifications} we recall the definition of Whitney stratification
and then in \S\ref{s:generic} we prove that the critical locus of a generic map admits a natural stratification.
In \S\ref{s:eliashbergs-theorem}-\ref{s:gromovs-theorem} we recall the results
from \cite{eliashberg} and \cite{eliashberg-mishachev-holonomic}
which will be used in the proof of Theorem~\ref{th:h-principle}.
In \S\ref{s:local-global-equivalences} we give a few definitions of equivalence of germs
which are discussed in detail in Appendix~\ref{s:appendix-equivalences}.

In \S\ref{s:stratified-immersions}-\ref{s:if-locally-then-stratumwise}
we prove some technical properties of local behavior of generic maps near the strata
which we need to prove that locally $L$-equivalent generic germs at a stratum of the critical set are globally $L$-equivalent.
In \S\ref{s:wccg}-\ref{s:ttb} we construct the twisted tangent bundle
and then in \S\ref{s:h-principle-local} we prove Theorem~\ref{th:h-principle}.

In \S\ref{s:3-bundles} we discuss the classification of rank $3$ vector bundles over $3$-manifolds.
Then in \S\ref{s:3-germs}-\ref{s:3-w-classes} we give a combinatorial description of the twisted tangent bundle for $n=3$ and compute its Stiefel-Whitney classes.
Finally, in \S\ref{s:three-manifolds} we deduce Theorem~\ref{th:three-manifolds}.

\subsection*{Acknowledgments}
I wish to thank Alexey Gorinov for helpful suggestions and conversations.
Also I am grateful to 
Maxim Kazarian, 
Sergey Melikhov, 
Sergey Natanzon, 
Petr Pushkar, 
Victor Vassiliev and 
Vladimir Zhgun 
for useful discussions and advice.


\section{Preliminaries}\label{s:preliminaries}

\subsection{Notation}

Suppose we are given a topological space $X$ and a subset $A\subset X$.
When we write ``a neighborhood $A\subset U\subset X$'',
this means that $U$ is an open neighborhood of the subset $A$ in the ambient space $X$.
We use the M.\,Gromov's notation $\Op(A)$ for a ``sufficiently small neighborhood of $A$''
(see e.\,g.\ \cite{gromov-book} or \cite{eliashberg-mishachev-book}).


When we consider a collection of objects $A_i$, where $i$ belongs to some set of indices $\kappa$,
this is always written as $\{A_i\}$ instead of $\{A_i\}_{i\in\kappa}$.
In most situations the index $i$ takes values in $\Z_{\ge0}$ or in $\{0,1,\ldots,k\}$ for some $k$.
We specify the set of values of the index when this is necessary.

Suppose we have a set $X$ and a collection $\{X_i\}$ of its subsets.
We say that $\{Y_i\}$ is {\it a~refinement} of $\{X_i\}$,
if every $Y_i$ is a~subset of certain $X_j$
and every $X_i$ can be presented as a union $Y_{i_1}\cup Y_{i_2}\cup\ldots$
When we consider {\it covers} of a fixed subset $Z\subset X$
(i.\,e.\ $Z\subset\bigcup X_i$),
then $\{Y_i\}$ is called {\it a refinement} of $\{X_i\}$,
if every $Y_i$ is a subset of certain $X_j$
and $Z\subset\bigcup Y_i$.

Covers always supposed {\it open} unless we explicitly specify otherwise.


\subsection{Transversality and stratifications}\label{s:stratifications}


First we recall main definitions.

Let $X_1,X_2,Y$ be manifolds of arbitrary dimensions.
The maps $f_1:X_1\to Y$ and $f_2:X_2\to Y$ are called {\it transversal}
if for every $x_1\in X_1$ and $x_2\in X_2$ such that $f_1(x_1)=f_2(x_2)=y$, we have $\Im(df_1(x_1))+\Im(df_2(x_2))=T_yY$.
Similarly, submanifolds of $Y$ are {\it transversal} if their imbedding maps are transversal.

Note that by the implicit function theorem the intersection of transversal submanifolds is a smooth submanifold.
Let $E\to X$ be a vector bundle over a manifold.
A collection of sections $\{f_i:X\to E\}$ 
is said to be {\it in general position},
if every $f_i$ is transversal to the intersection of any collection of $f_j(X),\ j\ne i$.

\vspace{.5em}

A {\it stratified subset} of a manifold $X$ is a locally finite
disjoint union of submanifolds (called {\it strata}) such that
the following three conditions are satisfied.
The first condition is called {\it the frontier condition}, we denote it by (fr).
Conditions (a) and (b) are called Whitney regularity conditions,
they were introduced in \cite{whitney-local-properties}.
\begin{itemize}
\item[(fr)] The closure of every stratum
is a union of this stratum and certain strata of strictly smaller dimension.

\item[(a)] For every strata $C,C'$,
if a sequence $\{x_i\},\ x_i\in C,$ converges to a point $y\in C'$
and $\lim (T_{x_i}C)=T\subset T_yX$, then we have $T\supset T_yC'$.

\item[(b)] For every strata $C,C'$,
if sequences $\{x_i\},\ x_i\in C,$ and $\{y_i\},\ y_i\in C',$
both converge to a point $y\in C'$,
for lines $\overline{x_iy_i}$ (in some local coordinates near the point $y$)
we have $\lim \overline{x_iy_i}=l$
and $\lim (T_{x_i}C)=T\subset T_yX$, then we have $T\supset l$.
\end{itemize}
Note that condition (b) implies conditions (a) and (fr), see e.\,g.\ \cite{trotman-theory}.
The properties of stratified subsets are discussed also in~
\cite[Part\,I,~\S1]{goresky-macpherson}.

\vspace{.5em}

Let $X, X'$ be manifolds and let $Y\subset X$ be a stratified subset.
A map $f:X'\to X$ is called {\it transversal to $Y$} if it is transversal to every stratum.
By the Thom transversality theorem such maps form a dense open subset in $C^\infty(X',X)$.
If $f$ is transversal to $Y$,
then $f^{-1}(Y)\subset X'$ is a stratified subset
whose strata are the preimages of the strata of $Y$
(see e.\,g.~\cite[Part\,I,~\S1.3]{goresky-macpherson}).
In particular, 
by the implicit function theorem,
these preimages are smooth submanifolds of~$X'$.


\subsection{The stratification of the jet space}\label{s:generic}

We need to define generic maps so that
the set of the critical points of a generic map
has a stratification which is invariant under coordinate transformations.
%
It is sufficient to show that
the Thom-Boardman decomposition
admits a 
$\Diff$-invariant
refinement which is a stratification.
The existence of such a refinement is purely natural,
it was mentioned in 1973 by J.\,N.\,Mather in \cite[p.\,247]{mather-tb} without a proof.
Next we prove this assertion, see Theorem~\ref{th:sigma-stratification} below.

First let us prove the following proposition
which generalizes the well known result on stratifications of semialgebraic sets \cite[Part\,I,~Th.\,1.7]{goresky-macpherson}.
Recall, a subset $Y\subset\R^m$ is called {\it semialgebraic}
if $Y$ is the union of sets given by a finite system of polynomial (in)equations,
i.\,e.\ of sets of the form
$$\{x\in\R^m\ |\ f_1(x)=0,\ldots,\ f_k(x)=0,\ g_1(x)>0,\ldots,\ g_l(x)>0\},$$
where $f_i,g_i:\R^m\to\R$ are polynomial functions.
Note that the intersection, the union and the difference of semialgebraic sets are always semialgebraic sets.
The closure of every semialgebraic set is also a semialgebraic set (see e.\,g.\ \cite{coste-semialgebraic}).

\begin{proposition}\label{pr:diff-invariant-stratification-closed}
Suppose $X_1,\ldots,X_k\subset\R^m$ is a collection of semialgebraic sets, $X=\bigcup X_i$ is closed.
Suppose a group $G$ acts on $\R^m$ by polynomial diffeomorphisms preserving every $X_i$.
Then $X$ admits a stratification
such that every $X_i$ is a union of strata
and the action of $G$ preserves every stratum.
\end{proposition}

\begin{proof}
We will 
construct a stratification of $X$
by induction on $d=\max(\dim X_i)$.
If $d=0$, then $X$ is a discrete subset of $\R^m$.
Then we may take $X_i$ just to be our strata.

Suppose $d\ge 1$.
We may assume that the sets $X_i$ do not intersect each other.
Otherwise we replace the collection $\{X_i\}$ by the refinement
which sets are of the form
$$
\Big(\bigcap_{i\in A}X_i\Big)\setminus\Big(\bigcup_{i\notin A}X_i\Big)
\ \ \text{ for all $A\subset\{1,\ldots,k\}$}.
$$

Suppose $\dim X_1=\ldots=\dim X_s=d$ and $\dim X_{s+1}<d$, \ldots, $\dim X_k<d$.
Let $Y_i=X_i\setminus\Sing(X)$ and $Z_i=\overline{X_i}\setminus Y_i$ for $i=1,\ldots,s$.
Here $\Sing(X)$ denotes the set of singular points of $X$
and smooth points of $X$ with the dimension of the tangent space $<d$.
Clearly, $\overline{Y_i}$ do not cross $Y_j$ for any $i,j\le s$, $i\ne j$.
We set $Y_1,\ldots,Y_s$ to be $d$-dimensional strata of our stratification.

Take the collection
$\Omega=\{Z_1,\ldots,Z_s,\ X_{s+1},\ldots,X_k\}$.
If needed we can use the operation of refinement mentioned above,
so further we may assume that the elements of $\Omega$ are also disjoint
(if some of them coincide, we take them only once).
Note that for every $V\in\Omega$ we have $\dim V<d$.

Finally, take any $Y_i$, and any $V\in\Omega$.
Note that either $\overline{Y_i}\cap V=\es$, or $\overline{Y_i}\cap V=V$.
Suppose the latter case takes place.
The set of points $y\in V$ at which the Whitney's regularity condition (b) does not hold
is a semialgebraic set which have {\it positive codimension} in $V$, see \cite{kaloshin} and \cite{trotman-existence}.
Denote it by $W$.
If $W\ne\es$, we remove $V$ from $\Omega$ and add $W$ and $V\setminus W$ to $\Omega$.
Then we repeat this procedure to $Y_i$ and $W$, etc.

After a finite number of steps
we eventually obtain a refinement of $\Omega$ such that for any $Y_i$
and any $Y\in\Omega$ the Whitney's regularity condition (b) holds.
Note that if we again replace $\Omega$ by any refinement, this property will not be violated.
So we can apply the induction hypothesis to $\Omega$ to obtain a stratification of $X$.

Recall that the action of $G$ preserves every $X_i$.
This action is polynomial, so it preserves also $\Sing(X)$
and all operations of subsets such as intersection, union etc.,
as well as the property that the Whitney's regularity condition (b) holds or not.
Therefore the resulting sets of all the described refinements
are also preserved by the action of $G$.
%
\end{proof}

\begin{theorem}\label{th:sigma-stratification}
For any pair of
manifolds $U,V$
the set $\Sigma(U,V)\subset J^r(U,V)$ admits a stratification
which is a refinement of the Thom-Boardman decomposition
and which is invariant under the right action of $\Diff(U)$ and the left action of $\Diff(V)$.

Such a stratification can be chosen
functorially with respect to open imbeddings.
More precisely, for open subsets $U'\subset U$ and $V'\subset V$ the stratification of
$\Sigma(U',V')$ is equal
to the intersection of $J^r(U',V')\subset J^r(U,V)$ with the stratification of $\Sigma(U,V)$.
\end{theorem}

\begin{proof}
Take $x\in U$ and $y\in V$.
By choosing local coordinates we can identify the space of $r$-jets $J^r_{x,y}(U,V)$
of maps that take $x$ to $y$
with $\R^m$ for some $m$.

For every sequence $I$ of length $r$
the intersection $\Sigma^I(U,V)\cap J^r_{x,y}(U,V)\subset\R^m$
is a variety (\cite[\S2]{boardman}; see also \cite[\S2.5]{avgz}).
Let $X=\Sigma(U,V)\cap J^r_{x,y}(U,V)$ and $X^I=\Sigma^I(U,V)\cap J^r_{x,y}(U,V)$.

The groups $\Diff_x(U)$ and $\Diff_y(V)$
of diffeomorphisms of $U$, respectively $V$, which preserve $x$, respectively $y$,
act on $J^r_{x,y}(U,V)$ by polynomial diffeomorphisms.
These actions preserve all $X^I$.
Then we can apply Proposition~\ref{pr:diff-invariant-stratification-closed} to
$\R^m=J^r_{x,y}(U,V)$, the collection of semialgebraic sets $\{X^I\}$
and the action of the group $G=\Diff_y(V)\times\Diff_x(U)$.
We obtain a $G$-invariant stratification of $\Sigma(U,V)\cap J^r_{x,y}(U,V)$.
Every stratum of the stratification lies in some Thom-Boardman submanifold.

Then using the action of $\Diff(V)\times\Diff(U)$
we can transfer this stratification from $J^r_{x,y}(U,V)$ to $J^r_{x',y'}(U,V)$
for all $(x',y')\in U\times V$.
As the original stratification is $G$-invariant, the result is well defined.
We define the strata of $\Sigma(U,V)$ as the images of the strata of $X$ under such identifications for all $(x',y')$.

We have naturally constructed the stratifications on the fibers of the bundle $J^r(U,V)\to U\times V$.
The space $J^r(U',V')$ can be viewed as the restriction of this bundle
to $U'\times V'\subset U\times V$, so the functoriality follows.
\end{proof}

Recall from \S\ref{s:thom-boardman-generic} that a map $\phi:U\to V$ is called {\it generic}
if and only if $j^r(\phi)$ is transversal to every stratum of $\Sigma(U,V)$.

\subsection{The $h$-principle for maps with fold singularities}\label{s:eliashbergs-theorem}

In the proof of Theorem~\ref{th:h-principle} we will use the following theorem from \cite[Th.\,2.2]{eliashberg}.
Let us briefly recall the notation.

If $p:E\to X$ and $p':E'\to Y$ are vector bundles,
$f:X\to Y$ is a continuous map and $F:E\to E'$ is a fiberwise map such that $p'\circ F=f\circ p$,
then we say that {\it $F$ covers $f$}.

Suppose 
$C\subset M$ is a codimension 1 closed submanifold without boundary.
A {\it $C$-immersion $M\to N$} is a map with fold singularities $\Sigma^{1,0}$ along $C$ and no other critical points.
A fiberwise map of tangent bundles $F:TM\to TN$ is called a {\it $C$-monomorphism} if
the restrictions $F|_{T(M\setminus C)}$ and $F|_{TC}$ are fiberwise injective
and there exist an involution $h:\Op(C)\to\Op(C)$
with the set of fixed points equal $C$
and such that over $\Op(C)$ we have $F\circ dh=F$.

Let $D\subset M$ be a closed $n$-submanifold  with boundary
such that every component of $M\setminus D$ intersects $C$ and let $f_0:D\to N$ be a smooth map.
Denote the space of maps $f:M\to N$ such that
$f|_D\equiv f_0$ and $f|_{M\setminus D}$ is a $C$-immersion by $\mathrm{Imm}_{C,D}(M,N)$.
Let $\mathrm{Mon}_{C,D}(TM,TN)$ be the space of fiberwise maps $F:TM\to TN$ such that
$F|_{TD}\equiv df_0$ and $F|_{T(M\setminus D)}$ is a $C$-monomorphism.

\begin{theorem}[Eliashberg]\label{th:eliashberg}
The map
$\pi_0\big(\mathrm{Imm}_{C,D}(M,N)\big) \to \pi_0\big(\mathrm{Mon}_{C,D}(TM,TN)\big)$
induced by taking the differential is surjective.
\hfill$\square$
\end{theorem}

In other words, every $F\in\mathrm{Mon}_{C,D}(TM,TN)$
can be deformed to the differential of a $C$-immersion
in the class of $C$-monomorphisms fixed over $D$.
In particular, if $\mathrm{Mon}_{C,D}(TM,TN)$ is nonempty, then $\mathrm{Imm}_{C,D}(M,N)$ is nonempty.

\subsection{The $h$-principle for immersions}\label{s:gromovs-theorem}

We will also use the following theorem, known as the relative Gromov's $h$-principle for immersions of open manifolds.
This result generalizes Smale's and Hirsch's theorems on immersions from \cite{hirsch}, \cite{hirsch-open} and \cite{smale}.
(Here we state only the partial case of this theorem that we will use.)

Let $U,V$ be $n$-manifolds with $U$ noncompact.
Suppose we are given a continuous map  $f:U\to V$ and
a fiberwise isomorphism $F:TU\to TV$ covering $f$.
Suppose there is a compact $n$-submanifold with boundary $D\subset U$
such that
the restriction $f|_D$ is smooth and $F|_D=df|_D$,
and such that for some polyhedron $K\subset U$ of positive codimension
$U$ can be compressed into an arbitrarily small neighborhood of $K\cup D$ by an isotopy fixed on $D$.

\begin{theorem}[Gromov]\label{th:gromov}
There exists an immersion $f':U\to V$ such that
$F$ is homotopic to $df'$ in the class of fiberwise isomorphisms fixed over~$D$.
\hfill$\square$
\end{theorem}

This theorem can be deduced by using the holonomic approximation theorem, see 
\cite[Th.\,1.2.1]{eliashberg-mishachev-holonomic}.
See also \cite[\S3]{eliashberg-mishachev-book} or \cite[\S1.1.3]{gromov-book}
for more details.

\subsection{Local and global quivalences of germs}
\label{s:local-global-equivalences}

Here we recall the definitions of local and global ($L$-)equivalences of germs.

Let $S\subset M$ be a locally closed subset of a manifold.
Suppose we are given maps $\phi_1:U_1\to V_1$ and $\phi_2:U_2\to V_2$
from neighborhoods $U_1,U_2\supset S$ to some $n$-manifolds $V_1,V_2$.

\begin{definition}\label{def:global-equivalence}
The $S$-germs of $\phi_1$ and $\phi_2$ are called {\it globally equivalent}
if there are a neighborhood $U_3\supset S$, an $n$-manifold $V_3$, a map $\phi_3:U_3\to V_3$,
immersions $\alpha_1:U_3\to U_1$ and $\alpha_2:U_3\to U_2$ restricting to the identity on $S$,
and immersions $\beta_1:V_3\to V_1$ and $\beta_2:V_3\to V_2$
such that the diagram 
$$
\CD
  U_1 @<\alpha_1<< U_3 @>\alpha_2>> U_2 \\
  @V \phi_1 VV @V \phi_3 VV @V \phi_2 VV  \\
  V_1 @<\beta_1<< V_3 @>\beta_2>> V_2
\endCD
$$
commutes.

The $S$-germs of $\phi_1$ and $\phi_2$ are called {\it globally $L$-equivalent}
if they are globally equivalent, and in addition to that $U_3$ can be taken to be a subset of $U_1\cap U_2$
and the maps $\alpha_1,\alpha_2$ can be taken to be the inclusions.

\end{definition}


\begin{definition}\label{def:local-equivalence}
The $S$-germs of $\phi_1$ and $\phi_2$ are called {\it locally equivalent}
if for any point $x\in S$ the germs of $\phi_1$ and $\phi_2$ at $x$ are globally equivalent.

The $S$-germs of $\phi_1$ and $\phi_2$ are called {\it locally $L$-equivalent}
if for every point $x\in S$ the germs of $\phi_1$ and $\phi_2$ at $x$ are globally $L$-equivalent.

\end{definition}

We discuss these definitions in details in Appendix~\ref{s:appendix-equivalences}.
In particular, we will prove that they are actually equivalence relations
and that locally ($L$-)equivalent germs may not be globally ($L$-)equivalent.

\section{Local behavior of generic maps}\label{s:local}

In this section we study some properties of generic maps which we need to prove Theorem~\ref{th:h-principle}.
It turns out that the properties hold for a wider class of maps.

\subsection{Stratified immersions}\label{s:stratified-immersions}




Let $X$ be a topological space.
We say that a subset $A\subset X$ {\it has dense interior},
if $\overline{\int(A)}\supset A$.

\begin{proposition}\label{pr:closure-of-interior}
Let $V$ and $V'$ be $n$-manifolds. 
Suppose a subset $R\subset V$ has dense interior.
Let $\alpha,\beta:V\to V'$ be maps such that $\alpha|_R\equiv\beta|_R$.
Then $d\alpha|_R \equiv d\beta|_R$.
\end{proposition}

\begin{proof}
Note that $d\alpha|_x = d\beta|_x$ for every $x\in \int(R)$.
Then the statement of the proposition directly follows from the continuity of the derivatives.
\end{proof}

\begin{definition}\label{def:stratified-immersion}
Let $U,V$ be $n$-manifolds.
A map $\phi:U\to V$ is called a {\it stratified immersion},
if there is a stratification of $U$ such that the restriction of $\phi$ to any stratum is an immersion.

The union of strata of dimension $<n$ is called {\it the skeleton} of the stratified immersion.
Note that the skeleton can be viewed as a subcomplex of positive codimension in $U$
(one can use e.\,g.\ the results of \cite{goresky} to show that).
\end{definition}

\begin{remark}\label{r:generic-is-stratified-immersion}
Any generic map of $n$-manifolds is a stratified immersion.
Indeed, we can define a dimension $n$ stratum of the source manifold as the set of regular points,
and then define a stratification of the set of critical points using Theorem~\ref{th:sigma-stratification}.
\end{remark}

\begin{proposition}\label{pr:image-of-generic}
Let $U$, $V$ and $V'$ be $n$-manifolds.
Suppose $\phi:U\to V$ is a stratified immersion
and $\alpha,\beta:V\to V'$ are maps such that $\alpha\circ\phi=\beta\circ\phi$.
Then $d\alpha|_x = d\beta|_x$ for every $x\in \phi(U)$.
\end{proposition}

\begin{proof}
Denote the complement of the skeleton of $\phi$ by $U^\circ$.
Recall that $\phi|_{U^\circ}$ is an immersion, so all points of $U^\circ$ are regular for $\phi$.
Therefore all points of $\phi(U^\circ)$ are interior points of $\phi(U)$.

Also note that $U^\circ$ is dense in $U$,
so, by continuity, $\phi(U^\circ)$ is dense in $\phi(U)$.
We have shown that $\phi(U)$ has dense interior.
It remains to apply Proposition~\ref{pr:closure-of-interior}.
\end{proof}

\subsection{Normal form of a stratified immersion}\label{s:normal-form}

Below we present some important properties of germs of stratified immersions at the strata.
Some of these properties in case of stable maps are well known (see e.\,g.\ \cite[\S9.5]{avgz}).
However, generic maps we consider may not be stable.

\begin{proposition}\label{pr:normal-form}
Suppose $U,V$ are $n$-manifolds, $C\subset U$ is a locally closed $k$-submanifold and $f:U\to V$ is a map such that $f|_C$ is an immersion.
Then there is a tubular neighborhood $U_C\supset C$ such that $f|_{U_C}$ is equal to the composition
$$U_C\xrightarrow{g}V_C\xrightarrow{h}V,$$
where $V_C$ is the total space of some rank $n-k$ vector bundle over $C$,
the neighborhood $U_C$ is viewed as 
the normal bundle to $C$,
a map $g$ is fiberwise and $h$ is an immersion.
More precisely, $g$ can be written as $(s,t)\mapsto(s,g_\nu(s,t))$,
where the first coordinate is a coordinate in $C$ and
the second coordinate is a local coordinate normal to $C$ in $U_C$, respectively in $V_C$.
\end{proposition}

\begin{proof}
For $V_C$ we take the pullback of the normal bundle of $f(C)$.
Let $h$ map it to some tubular neighborhood of $f(C)$.
Then $f$ restricted to some sufficiently small tubular neighborhood $U_C$ can be lifted to $V_C$.
So we have a map $g:U_C\to V_C$ and it remains to choose a vector bundle structure on $U_C$ such that $g$ is a fiberwise map.

Since $g|_C$ is a diffeomorphism onto the image,
for every $x\in C$ the image $dg(T_xU_C)$ is transversal to fibers of $V_C$.
By continuity of the differential, the same holds for $x$ close to $C$.
In other words, we can replace the neighborhood $U_C$ by a smaller one so that
for any point $x\in U_C$ the image $dg(T_xU_C)$ is transversal to the fiber of $V_C$.

Then by the implicit function theorem the preimage of any fiber of $V_C$
is a smooth submanifold of $U_C$ transversal to $C$.
If needed, replace $U_C$ by a smaller neighborhood one more time
to make these preimages diffeomorphic to $\R^{n-k}$ and form a vector bundle over $C$.
\end{proof}

\begin{definition}
In the notation of Proposition~\ref{pr:normal-form},
we call the map $g:U_C\to V_C$ a {\it normal form of~$f$ at $C$}.
If a map is its own normal form at $C$, we call such a map {\it $C$-normal}.

Note that the $C$-germ of a normal form at $C$ is globally equivalent to the $C$-germ of the original map.
\end{definition}

Note that a normal form is not unique.
However, in view of the following Proposition~\ref{pr:normal-form-universal-property}
a normal form is unique up to global $L$-equivalence.

\begin{lemma}\label{l:generic-normal-form}
A stratified immersion has a normal form at every stratum. 
\end{lemma}

\begin{proof}
This directly follows from Definition~\ref{def:stratified-immersion} and Proposition~\ref{pr:normal-form}.
\end{proof}

In particular, by Remark~\ref{r:generic-is-stratified-immersion}
a generic map $f$ of $n$-manifolds has normal form at every stratum of $\Sigma(f)$.
Moreover, if we take the fibers $U_s\subset U_C$ and $V_s\subset V_C$ over a point $s\in C$
(using the notation of Proposition~\ref{pr:normal-form}),
then the map $g_\nu|_{U_s}:U_s\to V_s$ is generic.
This fact, which we will not use, can be
proved using some ``normal modification'' of 
Theorem~\ref{th:sigma-stratification}.

\subsection{The universal property for a normal form}

The following proposition is some kind of universal property for a normal form.

\begin{proposition}\label{pr:normal-form-universal-property}
Let $U,V_1,V_2$ be $n$-manifolds, let $C\subset U$ be a submanifold and
let $f_1:U\to V_1$ and $f_2:U\to V_2$ be maps whose $C$-germs are globally $L$-equivalent.
Suppose $f_1$ is $C$-normal.
Then there are neighborhoods $C\subset U'\subset U$ and $f(U')\subset V'_1\subset V_1$
and an immersion $\beta:V'_1\to V_2$ such that $\beta\circ f_1|_{U'}=f_2|_{U'}$.
\end{proposition}

First let us prove the following assertion.

\begin{proposition}\label{pr:diffeomorphism-to-the-image-tubular-neighborhood}
Let $U,V$ be $n$-manifolds,
let $\phi:U\to V$ be an immersion
and let $C\subset U$ be a locally closed submanifold of positive codimesnion
such that $\phi|_C:U\to V$ is an imbedding as a submanifold.
Then there is a neighborhood $C\subset U'\subset U$
such that the restriction $\phi|_{U'}$ is a diffeomorphism onto the image.
\end{proposition}

\begin{proof}
Choose a tubular neighborhood $\phi(C)\subset V'\subset V$
and fix the structure of a normal bundle on $V'$.
The next reasoning is similar to the proof of Proposition~\ref{pr:normal-form}.
Namely, choose a tubular neighborhood $C\subset U'\subset U$ with a vector bundle structure
so that $\phi$ maps $U'$ to $V'$ fiberwise.
If $U'$ is sufficiently close to $C$, then the restriction of $\phi$ to each fiber is injective.
Therefore $\phi|_{U'}$ is itself injective.

Note that one can give another proof of this proposition,
similar to the proof of Lemma~\ref{l:diffeomorphism-to-the-image-locally-closed} from \S\ref{s:stratumwise-equivalence}.
\end{proof}

\begin{proof}[Proof of Proposition~\ref{pr:normal-form-universal-property}]
We will use only that $f_1|_C:C\to V_1$ is an imbedding as a submanifold,
the fiberwise structures on $U$ and $V_1$ are not necessary.

By definition, if $C$-germs of $f_1$ and $f_2$ are globally $L$-equivalent,
there exist
a neighborhood $C\subset U''\subset U$,
an $n$-manifold $V_3$,
a map $f_3:U''\to V_3$
and immersions $\beta_1:V_3\to V_1$ and $\beta_2:V_3\to V_2$
such that $\beta_1\circ f_3=f_1|_{U''}$ and $\beta_2\circ f_3=f_2|_{U''}$.

It is easy to see that $f_3(C)\subset V_3$ is a submanifold
and the immersion $\beta_1|_{f_3(C)}$ is a diffeomorphism onto the image.
Therefore for some neighborhood $f_3(C)\subset V_3'\subset V_3$
the restriction $\beta_1|_{V_3'}$ is a diffeomorphism onto the image.
Then it remains to set $U'=f_3^{-1}(V_3')$, $V_1'=\beta_1(V_3')$
and $\beta=\beta_2\circ(\beta_1|_{V_3'})^{-1}$.
\end{proof}

\subsection{Cylindric covers}

Let $C$ be a manifold and $p:U\to C$ and $q:V\to C$ be vector bundles.
Let $\phi:U\to V$ be a fiberwise map such that $p=q\circ\phi$.
Take any subset $Z\subset C$.

\begin{definition}\label{def:cylindric-covers}
Suppose we are given a cover $Z\subset\bigcup U_i\subset U$.
Denote the union $\bigcup U_i$ by $B$.
The cover is called {\it cylindric},
if $U_i=B\cap p^{-1}(p(U_i))$ for every $i$.

Given cylindric covers $Z\subset\bigcup U_i\subset U$ and $\phi(Z)\subset\bigcup V_i\subset V$,
the map $\phi$ is called {\it vertical} (with respect to the covers),
if $U_i=B\cap\phi^{-1}(V_i)$ for every $i$.
\end{definition}

\begin{remark}\label{r:cylindric-intersection}
Note that if $\phi$ is vertical, then $\phi(U_i)\subset V_i$ for every $i$.
Clearly, for any $i,j$ we have $f(U_i)\cap f(U_j)=f(U_i\cap U_j)$.
Moreover,
$$\phi(U_i)\cap V_j=\phi(U_i\cap U_j)$$
for any $i,j$.
This is the key property of cylindric covers and vertical maps.
\end{remark}

\begin{proposition}\label{pr:cylindric-refinement}
Let $C$ be a manifold and $p:U\to C$ and $q:V\to C$ be vector bundles.
Let $\phi:U\to V$ be a fiberwise map such that $p=q\circ\phi$.
We consider $C$ as the zero section of $U$.
Take a closed subset $Z\subset C$.

Suppose we are given locally finite covers
$Z\subset\bigcup U_i$ and $\phi(Z)\subset\bigcup V_i$
such that $\phi(U_i)\subset V_i$ for all $i$.
Then the covers admit cylindric refinements $\bigcup U'_i$ and $\bigcup V'_i$
such that the map $\phi$ is vertical.
Moreover, the refinements can be chosen so that
$U'_i\subset U_i$ and $V'_i\subset V_i$ for all $i$.
\end{proposition}

\begin{proof}
Choose norms on the fibers of $U$ and $V$.
We denote these norms similarly by~$\|\cdot\|$.
Given a fiber $F$ and a point $x\in F$,
we denote the open $\eps$-ball about $x$ in $F$ by $B_\eps(x)$.

Choose a cover $Z\subset\bigcup W_i\subset C$ such that $\overline{W_i}\subset U_i$ for all $i$.
Recall that $\phi(C)\subset V$ is the zero section and $\phi|_C$ is a homeomorphism.
Therefore $\overline{\phi(W_i)}=\phi(\overline{W_i})\subset V_i$ for all $i$.

\vspace{.5em}

To define $V'_i$, for each $i$ we take a function $\xi_i:\phi(C)\to \R_{\ge0}$ such that
$B_{\xi_i(x)}(x)\subset V_i$ for all $x\in\phi(C)$ and $\xi_i(x)\ne0$ for all $x\in V_i\cap \phi(C)$.
Then for each $i$ take any function $\eta_i:\phi(C)\to \R_{>0}$ such that
$\eta_i\equiv\xi_i$ on $\phi(\overline{W_i})$ and $\eta_i\equiv5$ on $\phi(C)\setminus V_i$.

Define a function $\eta:\phi(C)\to \R_{>0}$ as $\eta(x)=\min \eta_i(x)$.
Note that, since the cover $\bigcup V_i$ is locally finite, $\eta$ is well defined and continuous.
Then the set $B=\{y\in V\ |\ \|y\|<\eta(\phi(q(y)))\}$ is an open neighborhood of $\phi(C)$.

Finally, we take $V'_i=q^{-1}(W_i)\cap B$ for each $i$.
By the construction, $V'_i\subset V_i$ and $Z\subset\bigcup V'_i$.

\vspace{.5em}

To define $U'_i$ we replace $U_i$ by $U_i\cap\phi^{-1}(B)$
and apply to the obtained cover the procedure similar to the described above.
\end{proof}

\subsection{Building a global $L$-equivalence}\label{s:if-locally-then-stratumwise}


The following assertion will be very important in our proof of Theorem~\ref{th:h-principle}.

\begin{theorem}\label{th:if-locally-then-at-stratum}
Suppose we are given $n$-manifolds $U,V,V'$ and stratified immersions $\phi:U\to V$ and $\phi':U\to V'$.
Suppose for a locally closed submanifold $C\subset U$ the restrictions $\phi|_C$ and $\phi'|_C$ are immersions and $\phi$ is $C$-normal.
Suppose the $C$-germs of $\phi$ and $\phi'$ are locally $L$-equivalent.
Then there is an immersion $\beta:\Op(\phi(C))\to V'$ such that $\beta\circ\phi|_{\Op(C)}=\phi'|_{\Op(C)}$.
\end{theorem}

\begin{proof}
Let $C=\bigcup_{i,k}\Delta_i^k$ be a sufficiently small triangulation.
Here $\Delta_i^k$ denotes a $k$-simplex without boundary, $k=0,1,\ldots,m$, where $m=\dim(C)$.
Since the $C$-germs of $\phi$ and $\phi'$ are locally $L$-equivalent, we may assume that
for every $\Delta_i^k$ there are
neighborhoods $\Delta_i^k\subset U_i^k\subset U$ and $\phi(U_i^k)\subset V_i^k\subset V$
and an immersion $\beta_i^k:V_i^k\to V'$
such that $\beta_i^k\circ\phi|_{U_i^k}=\phi'|_{U_i^k}$.
The neighborhoods can be taken so that
the cover $\phi(C)\subset\bigcup V_i^k$ is locally finite
and $V_i^k\cap V_j^k=\es$ for $i\ne j$ and all $k$.

We will prove the lemma by induction.
On the $k$-th step we will obtain an immersion $\beta^k:\Op(\phi(\sk^k(C)))\to V'$
such that $\beta^k\circ\phi|_{\Op(\sk^k(C))}=\phi'|_{\Op(\sk^k(C))}$.
Let us describe the $k$-th step, $k=0,1,\ldots,m$.

By the induction hypothesis, there are
neighborhoods $\sk^{k-1}(C)\subset U^{k-1}\subset U$ and $\phi(U^{k-1})\subset V^{k-1}\subset V$
and an immersion $\beta^{k-1}:V^{k-1}\to V'$
such that $\beta^{k-1}\circ\phi|_{U^{k-1}}=\phi'|_{U^{k-1}}$.
(If $k=0$, we set $U^{k-1}$ and $V^{k-1}$ to be empty.)

We may assume that the covers
$\sk^k(C)\subset \big(U^{k-1}\cup\bigcup U_i^k\big)\subset U$ and
$\phi(\sk^k(C))\subset \big(V^{k-1}\cup\bigcup V_i^k\big)\subset V$
are cylindric and the map $\phi$ is vertical.
Otherwise we replace the sets $U^{k-1}$, $V^{k-1}$, $U_i^k$ and $V_i^k$
by smaller ones applying Proposition~\ref{pr:cylindric-refinement}.
Recall that
$$
\beta_i^{k}\circ\phi|_{U^{k-1}\cap U_i^k}=
\phi'|_{U^{k-1}\cap U_i^k}=
\beta^{k-1}\circ\phi|_{U^{k-1}\cap U_i^k}.
$$
Therefore, by Remark~\ref{r:cylindric-intersection},
for every $j$
the maps $\beta^{k-1}$ and $\beta_j^k$ coincides on
the intersection
$V^{k-1}\cap V_j^k\cap\phi\big(U^{k-1}\cup\bigcup U_i^k\big)$.

Take a neighborhood $\phi(\sk^{k-1}(C))\subset W\subset V$
such that $\overline{W}\subset V^{k-1}$.
For every $V_i^k$ we define a function $t:V_i^k\to[0;1]$ such that
$t=0$ on $V_i^k\setminus V^{k-1}$ and $t=1$ on $V_i^k\cap\overline{W}$.
Then we define a map $\tilde\beta_i^k:V_i^k\to V'$ as
$$\tilde\beta_i^k(x)=t\cdot\beta^{k-1}(x)+(1-t)\cdot\beta_i^k(x).$$
Here we use an Euclidean structure which comes from arbitrary local coordinates in $V$ near $\phi'(\Delta_i^k)$.
Since the triangulation of $C$ is sufficiently small, we may assume that such coordinates exist.
Note that the map $\tilde\beta_i^k$ depends on the choice of this coordinates, but we can take any of them.

By the construction, $\tilde\beta_i^k\equiv\beta_i^k$ on $\phi(U_i^k)$.
So Proposition~\ref{pr:image-of-generic} implies that $d\tilde\beta_i^k|_{\phi(\Delta_i^k)}$ has rank $n$.
Then, by continuity of the differential,
$\tilde\beta_i^k|_{\Op(\phi(\Delta_i^k))}$ is an immersion.
This immersion coincides with $\beta^{k-1}$ on $W$.
Together the maps $\beta^{k-1}$ and $\beta_i^k$ define a desired immersion $\Op(\phi(\sk^k(C)))\to V'$,
so the proof of $k$-th step is complete.
\end{proof}

\section{Global approach to generic germs}\label{s:global}

In this section we will recall the main definitions and then prove Theorem~\ref{th:h-principle}.

\subsection{Locally compatible maps}\label{s:wccg}

First let us recall the following definition.

\begin{definition}\label{def:cwcm}
Suppose we are given a collection of maps $\{\phi_i:U_i\to V_i\}$,
where $U_i\subset M$ are open subsets and $V_i$ are $n$-manifolds.
We say that $\{\phi_i\}$ is a {\it collection of locally compatible maps},
if for every $i,j$ the germs of $\phi_i$ and $\phi_j$ at $U_i\cap U_j$ are locally $L$-equivalent.

Two collections of locally compatible maps $\{\phi_i\}$ and $\{\phi'_i\}$
are called {\it locally $L$-equivalent}, if their union $\{\phi_i\}\cup\{\phi'_i\}$
is a collection of locally compatible maps.
A map $f:M\to N$ is {\it locally $L$-equivalent}
to the collection of of locally compatible maps $\{\phi_i\}$,
if
the collection $\{\phi_i\}$ and the collection consisting of a single map $\{f\}$ are locally $L$-equivalent.
\end{definition}

One may think of a collection of locally compatible maps as a global map on $\bigcup U_i$ with variable target.

\begin{proposition}\label{pr:cwcg-stratification}
Let $S\subset M$ be a closed subset and
let $\{\phi_i:U_i\to V_i\}$ be a collection of locally compatible generic maps
such that $\bigcup\Sigma(\phi_i)=S$.
Then there is a stratification of $S$ such that
for every stratum $C\subset S$ and for every $i$
the restriction $\phi_i|_{C\cap U_i}$ is an immersion.
\end{proposition}

\begin{proof}
Using Theorem~\ref{th:sigma-stratification} we choose a stratification of $\Sigma(\phi_i)=S\cap U_i$ for every $i$.
The obtained stratifications are agreed on the intersections $U_i\cap U_j$ by the naturalness and the $\Diff$-invariance.

Recall that every stratum $C\subset S$ lies in some $\Sigma^I(\phi_i)$,
where $I$ is a sequence of length $n+1$.
Since $j^{n+1}(\phi_i)$ is transversal to $\Sigma^I(U_i,V_i)$,
by dimension reason the sequence $I$ has zero at the end.
This means that the restriction $\phi_i$ to $\Sigma^I(\phi_i)$
(and, therefore, to $C\cap U_i$) is an immersion.

Note that the obtained stratification of $S$ is a canonical refinement of the Thom-Boardman decomposition,
since our operations in the proof of Proposition~\ref{pr:diff-invariant-stratification-closed} did not depend on the choice of local coordinates etc.
\end{proof}

\begin{lemma}\label{l:cwcg-stratification}
For $S$ and $\{\phi_i\}$ as above,
there is a triangulation of $S$
such that every its open simplex $\Delta$ is a smooth submanifold in $M$
and for every $i$ the restriction $\phi_i|_{\Delta\cap U_i}$ is an immersion.
\end{lemma}

\begin{proof}
First make a stratification of $S$ using Proposition~\ref{pr:cwcg-stratification},
and then apply the result of \cite{goresky}.
Note that, unlike the stratification of $S$,
our triangulation is not at all canonical.
\end{proof}

\begin{definition}
Suppose we are given two collections of locally compatible maps
$\{\phi'_i:U'_i\to V'_i\}$ and $\{\phi_i:U_i\to V_i\}$.
Then $\{\phi'_i\}$ is called {\it a refinement} of $\{\phi_i\}$,
if they are locally $L$-equivalent and every $U'_i$ lies in some $U_j$.
\end{definition}

\subsection{Twisted tangent bundle}\label{s:ttb}

Here we recall the definition from \S\ref{s:intro-ttb} and prove that it is correct.
Let us begin with the following important construction.

\begin{proposition}\label{pr:gluing-wcwm-bundle}
Let $\{\phi_i:U_i\to V_i\}$ be a collection of locally compatible generic maps.
Denote the vector bundles $\phi_i^*(TV_i)$ by $E_i$.
Then there are fiberwise isomorphisms $\Phi_{i,j}:E_i|_{U_i\cap U_j}\to E_j|_{U_i\cap U_j}$ such that
$\Phi_{i,i}=\Id_{E_i}$ and $\Phi_{k,i}\circ \Phi_{j,k}\circ \Phi_{i,j}|_{U_i\cap U_j\cap U_k}=\Id_{E_i}|_{U_i\cap U_j\cap U_k}$ for every $i,j,k$.
\end{proposition}

In fact, the constructed below fiberwise isomorphisms $\Phi_{i,j}$ are natural in the following sense:
each $\Phi_{i,j}$ is defined by the maps $\phi_i$ and $\phi_j$ and does not depend on $\phi_k$, $k\ne i,j$.

\begin{proof}
For every point $x\in U_i\cap U_j$ choose a neighborhood $x\in U_x\subset U_i\cap U_j$
and neighborhoods $\phi_i(U_x)\subset V_{i,x}\subset V_i$ and $\phi_j(U_x)\subset V_{j,x}\subset V_j$
with a diffeomorphism $\beta_x:V_{i,x}\to V_{j,x}$ such that $\beta_x\circ \phi_i|_{U_x}=\phi_j|_{U_x}$.
We set $\Phi_{i,j}$ to be equal $d\beta_x$ on the fiber of $E_i$ at $x$.
By Proposition~\ref{pr:image-of-generic} this does not depend on the choice of the diffeomorphism $\beta_x$.

By the same reasoning Proposition~\ref{pr:image-of-generic} implies that
$\Phi_{i,j}$ is well defined over the whole of $U_i\cap U_j$.
Note that $\Phi_{i,j}$ is continuous and smooth over a neighborhood of any point $x\in U_i\cap U_j$,
so it is continuous and smooth over the whole of $U_i\cap U_j$.
The equality $\Phi_{k,i}\circ \Phi_{j,k}\circ \Phi_{i,j}|_{U_i\cap U_j\cap U_k}=\Id_{E_i}|_{U_i\cap U_j\cap U_k}$
follows directly from Proposition~\ref{pr:image-of-generic} as well.
\end{proof}

Suppose we are given a closed subset $S\subset M$ 
and a collection of locally compatible generic maps $\{\phi_i:U_i\to V_i\},\ i=1,2,\ldots$
such that $\bigcup \Sigma(\phi_i)\subset S$.

\begin{definition}\label{def:twisted-tangent-bundle-cwcm}
The {\it twisted tangent bundle $T^{\{\phi_i\}}M$} is the following rank $n$ vector bundle over~$M$.
Let $U_{0}=V_{0}=M\setminus S$ and set $\phi_{0}:U_0\to V_0$ to be equal $\Id_{M\setminus S}$.
Note that $\{\phi_i\}$ together with $\phi_{0}$ is a collection of locally compatible generic maps.
Applying Proposition~\ref{pr:gluing-wcwm-bundle} for it,
we obtain the vector bundles $E_i$ over $U_i$ and the fiberwise isomorphisms $\Phi_{i,j}$ over $U_i\cap U_j$,
$i=0,1,\ldots$.
Then the total space of $T^{\{\phi_i\}}M$ is obtained by gluing together $\{E_i\}$ using $\{\Phi_{i,j}\}$.
\end{definition}


\begin{lemma}\label{l:ttb-well-defined}
We are given two collections of locally compatible generic maps
$\{\phi_i:U_i\to V_i\}$ and $\{\phi'_i:U'_i\to V'_i\}$
such that $\bigcup \Sigma(\phi_i)\subset S$ and $\bigcup \Sigma(\phi'_i)\subset S$.
Suppose $\{\phi_i\}$ and $\{\phi'_i\}$ are locally $L$-equivalent.
Then the twisted tangent bundles $T^{\{\phi_i\}} M$ and $T^{\{\phi'_i\}}M$ are isomorphic.
\end{lemma}

\begin{proof}
By definition, the union $\{\phi_i\}\cup\{\phi'_i\}$ is locally compatible.
To prove the proposition we just compose the tautological ``forgetful'' isomorphisms
$T^{\{\phi_i\}\cup\{\phi'_i\}} M\simeq T^{\{\phi_i\}} M$
and $T^{\{\phi_i\}\cup\{\phi'_i\}} M\simeq T^{\{\phi'_i\}} M$.
\end{proof}

\subsection{Proof of Theorem~\ref{th:h-principle}}\label{s:h-principle-local}

Recall that we are given $n$-manifolds $M,N$, $n>1$, a continuous map $f:M\to N$ and a~nonempty closed subset $S\subset M$.
Also we are given a collection of locally compatible generic maps $\{\phi_i:U_i\to V_i\}$
such that $\bigcup\Sigma(\phi_i)=S$.
We have to prove that $f$ is homotopic to a generic map $f'$ such that
$f'$ and $\{\phi_i\}$ are locally $L$-equivalent and $\Sigma(f')=S$
if~and only if the vector bundles $T^{\{\phi_i\}}M$ and $f^*(TN)$ are isomorphic.

\begin{proof}
Suppose $f$ is homotopic to a generic map $f'$
which is locally $L$-equivalent to $\{\phi_i\}$
and which has no critical points outside~$S$.
Then by Lemma~\ref{l:ttb-well-defined} we have $T^{\{\phi_i\}} M\simeq f'^*(TN)\simeq f^*(TN)$.

It remains to prove the ``if'' part of the theorem.

\vspace{.5em}

Suppose the bundles $T^{\{\phi_i\}}M$ and $f^*(TN)$ are isomorphic.
By Lemma~\ref{l:cwcg-stratification}
the subset $S\subset M$ is triangulable.
Let $S=\bigcup_{i,k}\Delta_i^k$ be a sufficiently small triangulation.
Here $\Delta_i^k$ denotes a $k$-simplex without boundary, $k=0,1,\ldots,n-1$,
and $i$ takes value in some (finite or countable) set of indices.
Every $\Delta_i^k\subset M$ is a smooth submanifold.

Since the triangulation is sufficiently small, we may assume that
there is a cover of $S$ by tubular neighborhoods of the simplices,
which is a
refinement
of $\{U_i\}$.
Therefore we can replace the collection of locally compatible maps $\{\phi_i\}$ by a refinement $\{\phi_i^k:U_i^k\to V_i^k\}$,
where every $U_i^k\supset\Delta_i^k$ is a tubular neighborhood.
By Lemma~\ref{l:generic-normal-form} we may assume that every map $\phi_i^k$ is $\Delta_i^k$-normal.
We take $U_i^k$ so close to $\Delta_i^k$ that
$U_i^k\cap U_j^l\ne\es$ if and only if
$\Delta_i^k\subset\partial\Delta_j^l$ or $\Delta_j^l\subset\partial\Delta_i^k$
for any
pair of different simplices.

Denote the vector bundle $T^{\{\phi_i^k\}}M$ by $E$.
Note that $E\simeq T^{\{\phi_i\}}M$ by Lemma~\ref{l:ttb-well-defined}.
Let $F:E\to TN$ be a fiberwise isomorphism covering $f$,
which exists by the hypothesis of the ``if'' part of the theorem.
Let $\Phi_i^k:E|_{U_i^k}\to TV_i^k$ be the tautological fiberwise isomorphism covering $\phi_i^k$, defined as in \S\ref{s:ttb}.

\vspace{.5em}

Next we will deform $f$ and $F$ in $n$ steps, from $0$ to $n-1$.
After the $l$-th step will we will get a collections of immersions
$\beta_i^k:V_i^k\to N$ for all $k\le l$ and all $i$.
The immersions $\beta_i^k$ must be compatible with deformed $f$ and $F$.
Namely,
$$f|_{\Op(\Delta_i^k)}=\beta_i^k\circ\phi_i^k|_{\Op(\Delta_i^k)}
\ \ \ \text{and} \ \ \
F|_{\Op(\Delta_i^l)}=d\beta_i^k\circ\Phi_i^k|_{\Op(\Delta_i^k)}
\eqno (*)$$


\vspace{.5em}

{\it Step 0.}
%
For every $i$ we fix any immersion $\beta_i^0:V_i^0\to N$ such that
$(\beta_i^0\circ\phi_i^0)|_{\Delta_i^0}=f|_{\Delta_i^0}$ and
$d\beta_i^0\circ\Phi_i^0|_{\Delta_i^0}=F|_{\Delta_i^0}$.
Take a closed tubular neighborhood $\Delta_i^0\subset\hat U_i^0\subset U_i^0$.
Since $\hat U_i^0$ deformation retracts onto $\Delta_i^0$,
we can homotope $f|_{\hat U_i^0}$ to $\beta_i^0\circ\phi_i^0|_{\hat U_i^0}$.
Then we extend the homotopy to the whole of $M$.
Outside $\hat U_i^0$ we allow $f$ to be not smooth.

During the homotopy of $f$
we make a covering deformation of $F$ in the class of fiberwise isomorphisms
to obtain $F|_{\hat U_i^0}=d\beta_i^0\circ \Phi_i^0|_{\hat U_i^0}$.
We can do this since $\hat U_i^0$ deformation retracts onto $\Delta_i^0$.
Make such deformations for all $i$ independently.
(To abbreviate notation, 
we always redenote the obtained maps by $f$ and $F$, respectively.)

\vspace{.5em}

{\it Step $k$.}
Take any $k$-simplex $\Delta_i^k$.
Denote $\Delta_i^k\cap\Op(\sk^{k-1}(S))$ by $C$.
As a result of the previous steps we have that
the $C$-germs of $f$ and $\phi_i^k$ are locally $L$-equivalent.
Then by Theorem~\ref{th:if-locally-then-at-stratum}
there is an immersion $\beta:\Op(\phi_i^k(C))\to N$
such that $\beta\circ\phi_i^k|_{\Op(C)}=f|_{\Op(C)}$.

Take a tubular neighborhood $\phi_i^k(\Delta_i^k\cap\Op(\sk^{k-1}(S)))\subset V'\subset V_i^k$
such that $\beta$ is defined on
$\overline{V'}$.
Denote $\phi_i^k(\Delta_i^k)\cup\overline{V'}$ by $K$.
Then we can define a map $\beta':K\to N$
as $f\circ(\phi_i^k)^{-1}$ on $\phi_i^k(\Delta_i^k)$ and as $\beta$ on $\overline{V'}$.
Note that $\beta'$ is continuous.
Since $V_i^k$ deformation retracts onto $K$, the map $\beta'$ can be extended from $K$ to the whole of $V_i^k$.
Then applying the M.\,Gromov's $h$-principle (Theorem~\ref{th:gromov})
to $\beta'$ and $F\circ(\Phi_i^k)^{-1}$
we homotope $\beta'$ to an immersion $\beta_i^k:V_i^k\to N$ relatively $\overline{V'}$.

Choose sufficiently small tubular neighborhoods
$\Delta_i^k\subset U\subset M$ and $\sk^{k-1}(S)\subset U'\subset M$.
Denote $\overline{U}\setminus\partial\Delta_i^k$ by $\hat U$.
We may assume that $\phi_i^k$ is defined on $\hat U$
and that $\phi_i^k(\hat U\cap\overline{U'})\subset \overline{V'}$,
see Fig.~\ref{fig:neighborhoods-in-uk}.
Let $H=\Delta_i^k\cup\big(\hat U\cap\overline{U'})$.
Also we may assume that $\hat U$ deformation retracts onto $H$.

\begin{figure}[h]
\center{\includegraphics{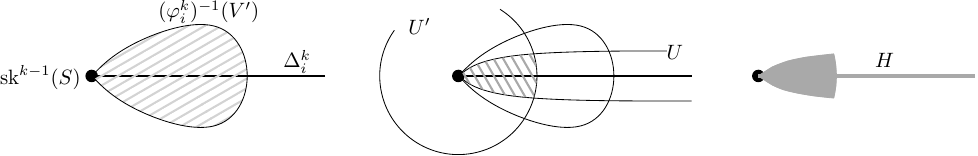}}
\caption{Choosing neighborhoods $U\supset\Delta_i^k$ and $U'\supset\sk^{k-1}(S)$.}\label{fig:neighborhoods-in-uk}
\end{figure}

Then we can homotope $f|_{\hat U}$ to $\beta_i^k\circ\phi_i^k|_{\hat U}$
relatively $\hat U\cap\overline{U'}$.
Extend this homotopy to the whole of $M$ relatively $\overline{U'}$.
At the same time we deform $F$ in the class of fiberwise isomorphisms fixed on $\overline{U'}$
so that as a result $F|_{\hat U}=d\beta_i^k\circ\Phi_i^k|_{\hat U}$.

We can apply the described deformations of $f$ and $F$
to all the $k$-simplices independently as they have disjoint neighborhoods.

\vspace{.5em}

{\it The final step.}
At the beginning of this step the map $f$ is smooth on $\Op(S)$ and
for all $k=0,\ldots,n-1$ and all $i$
there are immersions $\beta_i^k:V_i^k\to N$ such that
the equalities $(*)$ hold.

Define a fiberwise map $F':TM\to TN$ covering $f$ as
$$
F'(x)=
\left\{
\begin{array}{ll}
F(x) & \text{ if $x\in M\setminus S$,}\\
df(x) & \text{ if $x\in\Op(\Delta_i^k)$ for any $i,j$}.
\end{array}
\right.
$$
Here we identify $TM|_{M\setminus S}$ with $E|_{M\setminus S}$ as in Definition~\ref{def:twisted-tangent-bundle-cwcm}.
The map $F'$ is well defined and continuous in view of $(*)$ and Proposition~\ref{pr:gluing-wcwm-bundle}.

Take sufficiently small closed neighborhood $\sk^{n-2}(S)\subset D\subset M$.
We may assume that $M\setminus D$ is path-connected and $\sk^{n-2}(S)\setminus D\ne\es$.
Note that $F'|_{M\setminus D}$ is an $S$-monomorphism.
Applying the Eliashberg's $h$-principle (Theorem~\ref{th:eliashberg}) to $f$ and $F'$ relative $D$,
we deform $f$ to a smooth map such that $\Sigma(f)=S$ and
which is $L$-equivalent to $\{\phi_i\}$ as a collection of locally compatible maps.
\end{proof}

\section{Generic maps of 3-manifolds}\label{s:3-manifolds}

\subsection{Classifications of rank 3 vector bundles}\label{s:3-bundles}

As we saw in Theorem~\ref{th:h-principle}, the problem of homotopy of a given map to a map with prescribed singularities
is reduced to verifying of an isomorphism of certain vector bundles.
In dimension 3 the last problem is equivalent to checking the equality of Stiefel-Whitney characteristic classes.

\begin{lemma}\label{l:3-bundles-isomorphic}
Let $E_1\to X$ and $E_2\to X$ be rank 3 vector bundles over a $CW$-complex $X$ of dimension~3.
Then $E_1\simeq E_2$ if and only if $w_1(E_1)=w_1(E_2)$ and $w_2(E_1)=w_2(E_2)$.
\end{lemma}

\begin{proof}
Obviously, the characteristic classes of isomorphic vector bundles are equal.
So it remains only to prove the ``if'' part.

Suppose that $w_1(E_1)=w_1(E_2)$ and $w_2(E_1)=w_2(E_2)$.
Take the line bundle $L\to X$ such that $w_1(L)=w_1(E_1)$.
Then for $E_1'=E_1\otimes L$ and $E_2'=E_2\otimes L$ we have
$w_1(E_1')=w_1(E_2')=0$ and $w_2(E_1')=w_2(E_2')$.
Let us deduce that $E_1'\simeq E_2'$,
then the statement of the lemma will follow from the isomorphisms $E_1'\otimes L\simeq E_1$ and $E_2'\otimes L\simeq E_2$.

Denote the oriented grassmannian $(3,\infty)$ by $Gr$ and let $E\to Gr$ be the universal bundle.
We have to check that the maps $\eps_1:X\to Gr$ and $\eps_2:X\to Gr$ corresponding to $E_1'$, respectively $E_2'$,
are homotopic if $\eps_1^*w_2(E)=\eps_2^*w_2(E)$.

Since $Gr$ is $1$-connected, we may assume that $\eps_1|_{\sk^1(X)}\equiv\eps_2|_{\sk^1(X)}$.
By the obstruction theory,
the homotopy between $\eps_1$ and $\eps_2$ can be extended from $\sk^1(X)$ to $\sk^2(X)$ if and only if $o_1=o_2$,
where $o_1,o_2\in H^2(X;\pi_2(Gr))$ are the first obstructions for $\eps_1$, respectively $\eps_2$, to be null-homotopic
(for details see e.\,g.~\cite[\S8.4]{spanier}).

Note that by the Hurewicz theorem and the long exact sequence
we have $H_2(Gr;\Z)=\pi_2(Gr)=\pi_1(SO(3))=\Z_2$.
So by the universal coefficient theorem $H^2(Gr;\pi_2(Gr))=\Z_2$, denote its generator by $a$.
By the construction of $o_1$ and $o_2$ we have $o_1=\eps_1^*(a)$ and $o_2=\eps_2^*(a)$ (see \cite[\S8.4,~Th.\,2]{spanier}).

On the other hand, there are oriented rank 3 vector bundles with $w_2\ne0$
(e.\,g.\ the sum of two copies of the nonorientable line bundle over $\RP^2$ with the trivial line bundle),
so we have $w_2(E)\ne0$, therefore $a=w_2(E)$.
Thus $o_1=\eps_1^*w_2(E)=w_2(E_1')$ and $o_2=\eps_2^*w_2(E)=w_2(E_2')$.
So by the equality $w_2(E_1')=w_2(E_2')$ the maps $\eps|_{\sk^2(X)}$ and $\eps'|_{\sk^2(X)}$ are homotopic.

Finally, there is no obstruction to extension of the homotopy
between $\eps_1$ and $\eps_2$ from $\sk^2(X)$ to $\sk^3(X)$ since $\pi_3(Gr)=\pi_2(SO(3))=0$.
\end{proof}

\subsection{Generic germs in dimension 3}\label{s:3-germs}

Let $M,N$ be closed 3-manifolds.
Suppose we are given a generic map $f:M\to N$.
Then $\Sigma^{1}(f)\subset M$ is a closed 2-submanifold,
$\Sigma^{1,1}(f)\subset \Sigma^{1}$ is a closed 1-submanifold
and $\Sigma^{1,1,1}(f)\subset \Sigma^{1,1}$ is a discrete subset.

Recall that the points of $\Sigma^{1,0}(f)=\Sigma^{1}(f)\setminus \Sigma^{1,1}(f)$ are called {\it folds},
the germ of $f$ at each of these points is locally equivalent
to the germ of the map $\R^3\to\R^3,\ (x,y,z)\mapsto(x^2,y,z)$ at the origin.
The points of $\Sigma^{1,1,0}(f)=\Sigma^{1,1}(f)\setminus \Sigma^{1,1,1}(f)$ are called {\it cusps},
for them in some local coordinates $f$ can be written as $(x,y,z)\mapsto(x^3+xy,y,z)$.
And the points of $\Sigma^{1,1,1,0}(f)=\Sigma^{1,1,1}(f)$ are called {\it swallowtails},
the singularities of $f$ at these points are of the form $(x,y,z)\mapsto(x^4+x^2y+xz,y,z)$.

For every generic $f:M\to N$ the set $\Sigma(f)$ consists only of the three described strata (possibly, empty).
The germs of $f$ at different points of the same stratum are locally equivalent.
Also these singularities are stable.
This means that for every map $f':U\to V$,
which is sufficiently close to $f$ in $C^\infty$-topology,
its strata $\Sigma^{1,0}(f')$, $\Sigma^{1,1,0}(f')$ and $\Sigma^{1,1,1,0}(f')$
are close to $\Sigma^{1,0}(f)$, $\Sigma^{1,1,0}(f)$ and $\Sigma^{1,1,1,0}(f)$, respectively,
and the germs of $f'$ and $f$ at corresponding strata are equivalent.

\vspace{.5em}

\begin{figure}[h]
\center{\includegraphics{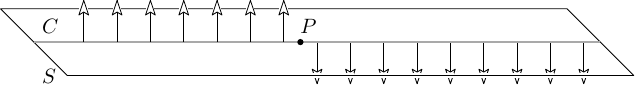}}
\caption{A characteristic vector field near a point of $P$.}\label{fig:characteristic-field}
\end{figure}

Suppose we are given a closed 2-submanifold $S\subset M$,
a closed 1-submanifold $C\subset S$ and a discrete subset $P\subset C$.

\begin{definition}\label{characteristic-field}
Given the triple $S,C,P$, a {\it characteristic field}
is a unit vector field in $M$ defined on $C\setminus P$ and normal to $S$.
This vector field must ``change the direction'' when we go through any point of~$P$, see Fig.~\ref{fig:characteristic-field}.

Suppose for a neighborhood $U\supset C$ and a $3$-manifold $V$ we have a generic map $\phi:U\to V$
such that $\Sigma^{1}(\phi)=S\cap U$, $\Sigma^{1,1}(\phi)=C$ and $\Sigma^{1,1,1}(\phi)=P$.
The {\it characteristic field of $\phi$} is the unit vector field normal to $S$
whose image is directed ``outside'' the cusp, see Fig~\ref{fig:characteristic-field-in-cusp}.
We denote it by $\nu(\phi)$.
\end{definition}

\begin{figure}[h]
\center{\includegraphics{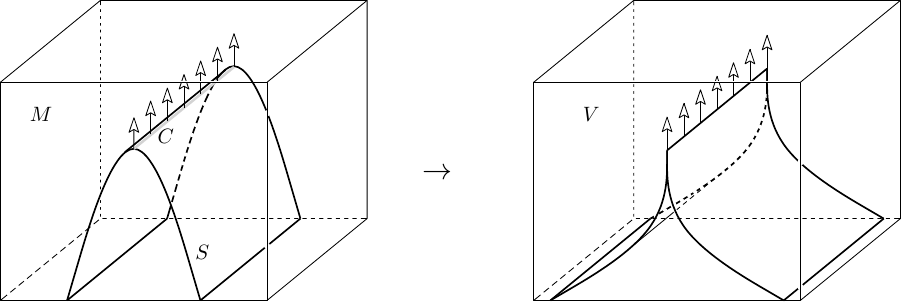}}
\caption{The characteristic field $\nu(\phi)$ and its image under $d\phi$.}
\label{fig:characteristic-field-in-cusp}
\end{figure}

Note that near a swallowtail point $\nu(\phi)$ has different directions on the two components of $C\setminus P$,
see Fig~\ref{fig:characteristic-field-in-swallowtail}.
So the characteristic field of a generic map is actually a characteristic field.

\begin{figure}[h]
\center{\includegraphics{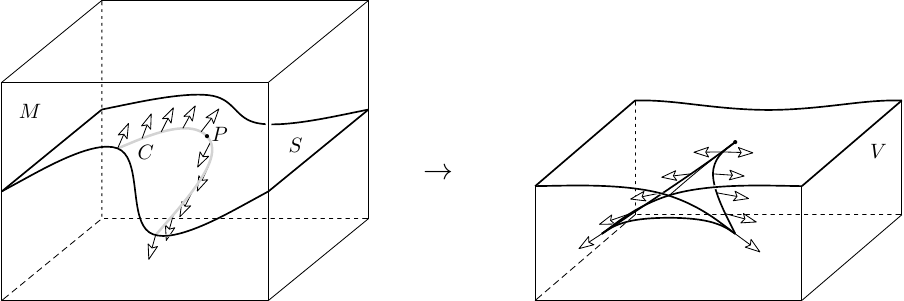}}
\caption{The characteristic field of a generic map near a swallowtail point and its image.}
\label{fig:characteristic-field-in-swallowtail}
\end{figure}

\begin{proposition}\label{pr:scpnu-exists}
Given a characteristic field $\nu_0$ and a tubular neighborhood $U\supset C$
one can construct a generic map $\phi:U\to V$ to some 3-manifold $V$
such that $\Sigma^1(\phi)=S\cap U$, $\Sigma^{1,1}(\phi)=C$, $\Sigma^{1,1,1}(\phi)=P$ and $\nu(\phi)=\nu_0$.
\end{proposition}

\begin{proof}
Take tubular neighborhoods $U_0\supset P$ and $U_1\supset C\setminus P$.
There are $n$-manifolds $V_0,V_1$
and generic maps $\phi_i:U_i\to V_i,\ i=0,1$,
which have the prescribed singularities and characteristic vectors at all points.   

Note that the germs of $\phi_0$ and $\phi_1$ at $(C\setminus P)\cap U_0$ are globally equivalent.
Composing $\phi_1$ with a suitable diffeomorphism of $U_1$,
we can make the germs of $\phi_0$ and $\phi_1$ at $(C\setminus P)\cap U_0$ globally $L$-equivalent.

Note that the germ of $\phi_1$ is $(C\setminus P)\cap U_0$-normal.
We set $V$ to be the result of gluing $V_0$ with $V_1$
via the map $\beta$ from Proposition~\ref{pr:normal-form-universal-property}
and let $U=U_0\cup U_1$.
If needed, we replace $U_0$, $U_1$, $V_0$ and $V_1$ by smaller ones,
so we may assume that $V$ is Hausdorff.
The maps $\phi_0$ and $\phi_1$ are compatible on $U_0\cap U_1$,
so together they define a desired map $\phi:U\to V$.
\end{proof}

\begin{proposition}\label{pr:swallowtails-parity}
There is a characteristic field for the triple $S,C,P$
if and only if for every component $C'\subset C$ we have $[C']\cdot[S]\equiv |P\cap C'|\mod2$.
\end{proposition}

\begin{proof}
Suppose there is a characteristic field $\nu$.
Passing a component $C'\subset C$ along $\nu$ we set it to the general position with $S$.
Then $C'\cap S=P\cap C'$.

To prove the ``if'' part
we slightly perturb every component $C'$ to obtain a new curve $C''\subset M$
which is transversal to $S$ and so that $C''\cap S=P\cap C'$.
Then
the arcs of $C''$ in $M\setminus S$
show the direction of a characteristic field.
\end{proof}

Note that if the property $[C']\cdot[S]\equiv |P\cap C'|\mod2$ holds for all components $C'\subset C$,
then the triple $S,C,P$ admits exactly two (opposite) characteristic fields.

Also note that for generic maps of surfaces the notion of characteristic vector is defined similarly,
see \cite[\S4.3]{eliashberg} or \cite[\S1]{ryabichev-surfaces}.

\subsection{Characteristic classes of the twisted tangent bundle}\label{s:3-w-classes}

If $n=2$ or $3$, then the twisted tangent bundle is well defined by a characteristic vector field.
More precisely, if two collections of locally compatible maps with the same singular loci
have the same characteristic vector fields,
then the corresponding twisted tangent bundles are isomorphic.
Next we will show how to compute their characteristic classes.

For generic maps of surfaces, the isomorphism class of the twisted tangent bundle
depends on the choice of characteristic vectors of the cusp points,
since the Euler class of the obtained bundle depends on this choice \cite[\S3]{ryabichev-surfaces}.

However, in dimension 3 the isomorphism class of the twisted tangent bundle depends only on the choice of $S$ and $C$,
in view of  Lemma~\ref{l:3-bundles-isomorphic}, Proposition~\ref{pr:scpnu-exists} and the following two assertions.


\begin{proposition}\label{pr:bundles-excision}
Let $E,E'\to X$ be rank $m$ vector bundles over an $n$-manifold $X$.
Suppose for a tubular neighborhood $U$ of a $k$-submanifold $Y\subset X$
the restrictions $E|_{X\setminus U}$ and $E'|_{X\setminus U}$ are isomorphic.
Then $w_i(E)=w_i(E')$ for all $i<n-k$.
\end{proposition}

\begin{proof}
%
%
Note that $\overline U$ deformation retracts onto $Y$.
Using Poincar\'e duality, for $i<n-k$ we have
$$H^i(\overline U,\partial\overline U)\simeq H_{n-i}(\overline U)\simeq H_{n-i}(Y)=0.$$
And by excision we have $H^i(\overline U,\partial\overline U)\simeq H^i(X,X\setminus U)$.

Since $H^i(X,X\setminus U)=0$ for every $i<n-k$,
the homomorphism $H^i(X)\to H^i(X\setminus U)$
induced by the inclusion $(X\setminus U)\subset X$ has zero kernel.
To complete the proof we apply the functoriality of $w_i$ to the last inclusion.
\end{proof}

Now let $M$ be a connected 3-manifold (for possibly disconnected $M$ the proof of the following lemma is similar).
Suppose we are given a closed 2-submanifold $S\subset M$ and a closed 1-submanifold $C\subset S$.
Take tubular neighborhoods $C\subset U_1\subset M$ and $(S\setminus C)\subset U_2\subset M$.

Then suppose there are generic maps $\phi_i:U_i\to V_i$ to some 3-manifolds $V_i$, $i=1,2$,
such that $\{\phi_1,\phi_2\}$ is a collection of locally compatible maps and
such that $\Sigma^1(\phi_2)=\Sigma^{1,0}(\phi_2)=S\setminus C$ and $\Sigma^{1,1}(\phi_1)=C$.
Denote the twisted tangent bundle $T^{\{\phi_1,\phi_2\}}M$ by $E$.

\begin{lemma}\label{l:w-of-t-phi-m}
We have $w_1(E)=w_1(M)+[S]$ and
$w_2(E)=w_2(M)+[S]^2+w_1(M)\cdot[S]+[C]$.
\end{lemma}

Note that, in particular, the characteristic classes do not depend on $\Sigma^{1,1,1}(\phi_1)$ and $\nu(\phi_1)$.

\begin{proof}
We define a vector bundle $E'$ over $M$ obtained by regluing $TM$ along $S$
using the orthogonal involution of $TM|_S$ which is identity on $TS$.
Note that the restrictions $E|_{M\setminus U_1}$ and $E'|_{M\setminus U_1}$ are isomorphic,
so by Proposition~\ref{pr:bundles-excision} we have $w_1(E')=w_1(E)$.
Next we will compute $w_1(E')$ and $w_2(E')$,
and then we will prove that $w_2(E)-w_2(E')=[C]$.

We cut $M$ along $S$ and denote the obtained manifold with boundary by $M'$.
Take the unit vector field on $\partial M'$ which is normal to $\partial M'$ and directed into $M'$.
It can be extend to a nonzero vector field $v$ on the complement in $M'$
of a finite collection of interior points $x_1,x_2,\ldots\in \int(M')$.
Denote $M\setminus{\{x_1,x_2,\ldots\}}$ by $M^\circ$.

Take the line subbundle of $TM'|_{M'\setminus\{x_1,x_2,\ldots\}}$ spanned by $v$.
We denote its projections to $TM|_{M^\circ}$ and to $E'|_{M^\circ}$ by $L$, respectively $L'$.
Note that, since $v$ is orthogonal to $\partial M'$, the subbundles $L$ and $L'$ are well defined over $S$.

We can see that $TM|_{M^\circ}\simeq L\oplus E_2$ and
$E'|_{M^\circ}\simeq L'\oplus E_2$ for some rank~$2$ vector bundle $E_2\to M^\circ$.
Then
$$w_i(L\oplus E_2)=\iota^*w_i(M)
\text{ \ and \ } w_i(L'\oplus E_2)=\iota^*w_i(E') \text{ \ for \ }i=1,2,\eqno (**)$$
where $\iota:M^\circ\to M$ is the imbedding.

Now we can compute $w_1(E')-w_1(M)$ and $w_2(E')-w_2(M)$.
By the construction
, the line bundle $L'$ is trivial, and $w_1(L)=\iota^*[S]$.
Therefore by $(**)$ we have
$$\iota^*w_1(E')-\iota^*w_1(M)=w_1(L')-w_1(L)=\iota^*[S].$$
Also by $(**)$ we have $\iota^*w_1(M)=w_1(L)+w_1(E_2)$,
therefore
$$\iota^*w_2(E')-\iota^*w_2(M)=
\big(w_1(L')-w_1(L)\big)\cdot w_1(E_2)=
\iota^*[S]\cdot w_1(E_2)=
\iota^*[S]^2+\iota^*[S]\cdot \iota^*w_1(M).$$
Thus we have shown that $w_2(E')-w_2(M)=[S]^2+[S]\cdot w_1(M)$.

\begin{figure}[h]
\center{\includegraphics{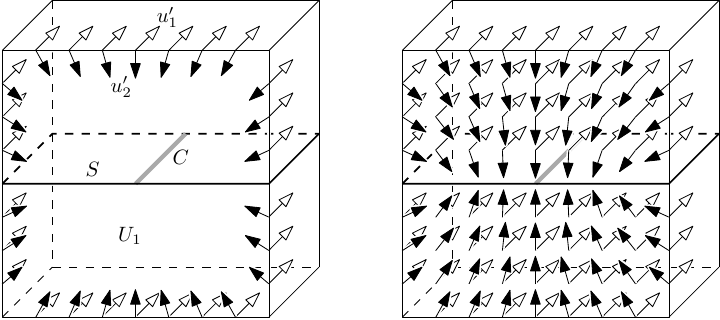}}
\caption{Two sections of $E'|_{\partial U_1}$ and their extensions to $U_1$.}\label{fig:cusp-sections}
\end{figure}

Now let us show that $w_2(E)-w_2(E')=[C]$.
Take two sections $u'_1,u'_2$ of $E'$ over $\partial U_1$ as in Fig.~\ref{fig:cusp-sections}, left
(here we identify $E'$ with $TM$ outside $S$).
Figure~\ref{fig:cusp-sections}, right shows that $u'_1,u'_2$ can be extended to the whole of $U_1$
as a pair of linearly independent sections.

Recall that $E|_{M\setminus U_1}\simeq E'|_{M\setminus U_1}$.
Let $u_1,u_2$ be the sections of $E$ over $\partial U_1$
corresponding to $u'_1|_{\partial U_1}$ and $u'_2|_{\partial U_1}$ under this isomorphism.
It is sufficient to show that general position extensions of $u_1,u_2$ to the whole of $U_1$
has set of linearly dependence $Z\subset U_1$ homologous to $C$.

\begin{figure}[h]
\center{\includegraphics{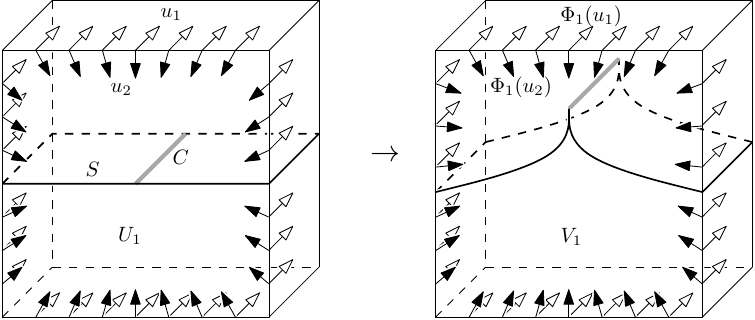}}
\caption{Two sections of $E|_{\partial U_1}$ and their images in $TV_1$.}\label{fig:cusp-sections-map}
\end{figure}

Take a fiber $F$ of the tubular neighborhood $U_1$ over any point of $C\setminus P$.
Apply the natural fiberwise isomorphism $\Phi_1:E|_{U_1}\to TV_1$ covering $\phi_1$
to the sections $u_1,u_2$.

Take the obtained vector fields $\Phi_1(u_1),\Phi_1(u_2)$ on $\partial\phi_1(F)$ (Fig.~\ref{fig:cusp-sections-map}, right)
and extend them as general position vector fields $v_1,v_2$ on the whole of $\phi_1(F)$,
which are linear dependent at one point $p$ (Fig.~\ref{fig:cusp-sections-target}).


Since $\Phi_1$ is a fiberwise isomorphism,
we have the well-defined map $\Phi_1^{-1}$, assigning to a section of $TV_1$ the corresponding section of $E|_{U_1}$.
Take the extensions $u_1,u_2$ to $F$ equal $\Phi_1^{-1}(v_1),\Phi_1^{-1}(v_2)$.
Note that these extensions are in general position and
they are linearly dependent at odd number of points (since $p$ has $3$ preimages).
Therefore $|Z\cap F|\equiv1\mod2$, so $[Z]=[C]$.
\end{proof}

\begin{figure}[h]
\center{\includegraphics{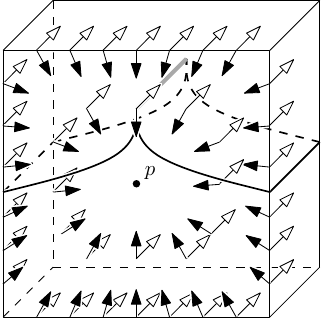}}
\caption{Vector fields $v_1,v_2$ which extends $\Phi_1(u_1),\Phi_1(u_2)$ to $\phi_1(F)$.}\label{fig:cusp-sections-target}
\end{figure}

\subsection{Proof of Theorem~\ref{th:three-manifolds}}\label{s:three-manifolds}

Here we apply the results of this section to deduce Theorem~\ref{th:three-manifolds} from Theorem~\ref{th:h-principle}.

Recall that we have closed 3-manifolds $M,N$,
a closed 2-submanifold $S\subset M$, a closed 1-submanifold $C\subset S$, a discrete subset $P\subset C$
and a continuous map $f:M\to N$.

We have to prove that the map $f$ is homotopic to a generic map $f'$ such that
$\Sigma^{1}(f')=S$, $\Sigma^{1,1}(f')=C$ and $\Sigma^{1,1,1}(f')=P$
if and only if
\begin{enumerate}
\item[(1)] $[S]=w_1(M)+f^*w_1(N)$,
\item[(2)] $[C]=w_1(M)\cdot[S]$ and
\item[(3)] for every component $C'\subset C$ we have $[C']\cdot[S]\equiv |P\cap C'|\mod2$.
\end{enumerate}

\begin{proof}
Suppose there is a generic map $f':M\to N$ with prescribed singularities and $f$ is homotopic to~$f'$.
Denote the twisted tangent bundle $T^{\{f'\}}M$ by $E$.

Let us prove conditions (1) and (2).
Lemma~\ref{l:w-of-t-phi-m} implies that
$w_1(E)=w_1(M)+[S]$ and $w_2(E)=w_2(M)+[S]^2+w_1(M)\cdot[S]+[C]$.
On the other hand, $E\simeq f^*(TN)$.
So $$f^*w_1(N)=w_1(M)+[S] \text{ \ \ and \ \ } f^*w_2(N)=w_2(M)+[S]^2+w_1(M)\cdot[S]+[C].$$
By the Wu formula (see e.\,g.\ \cite[p.\,131]{milnor-stasheff}),
$w_1(M)^2=w_2(M)$ and $w_1(N)^2=w_2(N)$.
Therefore,
%
\begin{multline*}
[C]\ =\
f^*w_2(N)+w_2(M)+[S]^2+w_1(M)\cdot[S]\ =\\
\big(f^*w_1(N)\big)^2+\big(w_1(M)\big)^2+[S]^2+w_1(M)\cdot[S]\ =\\
\big(w_1(N)+w_1(M)+[S]\big)^2+w_1(M)\cdot[S]\ =\
w_1(M)\cdot[S].
\end{multline*}

Finally, condition (3) follows from
Definition~\ref{characteristic-field} and Proposition~\ref{pr:swallowtails-parity}.

\vspace{.5em}

Suppose conditions (1-3) hold.
We take the collection of locally compatible maps $\{\phi_1,\phi_2\}$
as in the hypothesis of Lemma~\ref{l:w-of-t-phi-m}.
The existence of such a collection of locally compatible maps
follows from condition (3) and Propositions~\ref{pr:scpnu-exists} and \ref{pr:swallowtails-parity}.

Then by Lemmas~\ref{l:3-bundles-isomorphic} and \ref{l:w-of-t-phi-m} and conditions (1-2)
the vector bundles $T^{\{\phi_1,\phi_2\}}M$ and $f^*(TN)$ are isomorphic.
It remains to apply Theorem~\ref{th:h-principle}.
\end{proof}

\appendix

\section{Appendix. Equivalences of germs}\label{s:appendix-equivalences}

\subsection{Stratumwise equivalence}\label{s:stratumwise-equivalence}

In \S\ref{s:local-global-equivalences} we introduced
several notions of equivalence for germs
of maps $M\to N$ at a locally closed subset $S\subset M$.
In this article we always considered germs of stratified immersions at their skeleta,
so in 
our examples $S$ is a stratified subset of positive codimension in $M$
and the map we consider is a stratified immersion with respect to this stratification.

This leads us to another notion of equivalence for germs.

\begin{definition}\label{def:stratumwise-equivalence}
Assume that $S$ is a stratified submanifold.
The $S$-germs of $\phi_1$ and $\phi_2$ are called {\it stratumwise equivalent}
if for every stratum $C\subset S$ the $C$-germs of $\phi_1$ and $\phi_2$ are globally equivalent.

The $S$-germs of $\phi_1$ and $\phi_2$ are called {\it stratumwise $L$-equivalent}
if for every stratum $C\subset S$ the $C$-germs of $\phi_1$ and $\phi_2$ are globally $L$-equivalent.
\end{definition}


\begin{proposition}\label{pr:equivalences}
The relations introduced in Definitions
\ref{def:global-equivalence}, \ref{def:local-equivalence} and \ref{def:stratumwise-equivalence}
are actually equivalence relations.
\end{proposition}

In order to prove Proposition~\ref{pr:equivalences} we need the following lemma.

\begin{lemma}\label{l:diffeomorphism-to-the-image-locally-closed}
Suppose $S$ is a locally closed subset of an $n$-manifold $U$.
Let $\phi:U\to U$ be a map such that $\phi|_S=\Id_S$ and the differential $d\phi$ is an isomorphism at each point of $S$.
Then there is a neighborhood $S\subset U'\subset U$ such that $\phi|_{U'}$ is a diffeomorphism onto the image.
\end{lemma}

\begin{proof}
First we suppose that $S$ is compact.
Note that $d\phi$ is an isomorphism over $\Op(S)$.
Suppose there does not exist a neighborhood $U'$ as in the lemma.
Then for every $k=1,2,\ldots$ there are two points $x_k\ne y_k$
which are $<\frac1k$ away from $S$ and such that $\phi(x_k)=\phi(y_k)$.
By choosing subsequences we may assume that $\{x_k\}$ and $\{y_k\}$ are convergent.
Their limits lie in $S$.
If the limits coincide, then, since $d\phi$ is an isomorphism at this point, we have a contradiction with the inverse function theorem.
If the limits do not coincide, then we have a contradiction with injectivity of $\phi$ on $S$.

Now we suppose that $S$ is closed but possibly noncompact.
Then take a proper map $\xi:U\to\R_{\ge0}$.
For any $k\in\Z_{\ge0}$ the set $\xi^{-1}([k-1;k+2])\cap S$ is compact.
So there is a neighborhood $(\xi^{-1}([k-1;k+2])\cap S)\subset U_k\subset U$
such that $\phi|_{U_k}$ is a diffeomorphism onto the image.
We set
$$
W_k= U_{k-1}\cap U_k\cap U_{k+1}\cap\phi^{-1}(\xi^{-1}((k-\textstyle\frac12;k+\frac32)))
$$
and then let $U'=\bigcup W_k$.
Clearly, $U'\subset U$ is open.
Also $U'\supset S$ since $W_k\supset\xi^{-1}([k;k+1])\cap S$.
The differential $d\phi$ is an isomorphism on $U'$.
Let us show that $\phi|_{U'}$ is injective.

Take any $x\in W_k$.
Then $\xi(\phi(x))\in(k-\frac12;k+\frac32)$.
Therefore if $\phi(y)=\phi(x)$, then either $y\in W_{k-1}$, or $y\in W_k$, or $y\in W_{k+1}$.
In each of these cases we see that
$x,y\in U_k$, but
$\phi|_{U_k}$ is injective, so $y=x$.

Finally, suppose that $S\subset U$ is an arbitrary locally closed subset.
Then the set $\tilde U=U\setminus(\overline S\setminus S)$ is open in $U$.
It remains to apply the same reasoning as above to the closed subset $S\subset\tilde U$.
\end{proof}

\begin{proof}[Proof of Proposition~\ref{pr:equivalences}]
Reflexivity and symmetry are evident.
Next we will show transitivity for the global equivalence.
The proof for the other
five relations
is similar.

Suppose we are given maps $\phi_1:U_1\to V_1$, $\phi_2:U_2\to V_2$ and $\phi_3:U_3\to V_3$
where $U_1,U_2,U_3\supset S$ are neighborhoods in $U$ and $V_1,V_2,V_3$ are $n$-manifolds.
Suppose that the $S$-germ of $\phi_1$ is globally equivalent to the $S$-germ of $\phi_2$
and the latter is globally equivalent to the $S$-germ of $\phi_3$.
In more detail, we have the following commutative diagram:
$$
\xymatrix{
& U_{1,2}\ar@{->}[dl]_{\alpha_1}\ar@{->}[dr]^{\alpha_2}\ar@{->}[dd]_{\phi_{1,2}} &
& U_{2,3}\ar@{->}[dl]_{\alpha'_2}\ar@{->}[dr]^{\alpha'_3}\ar@{->}[dd]_{\phi_{2,3}} & \\
U_1\ar@{->}[dd]_(0.3){\phi_1} & &
U_2\ar@{->}[dd]_(0.3){\phi_2} & &
U_3\ar@{->}[dd]_(0.3){\phi_3} \\
& V_{1,2}\ar@{->}[dl]_{\beta_1}\ar@{->}[dr]^{\beta_2} &
& V_{2,3}\ar@{->}[dl]_{\beta'_2}\ar@{->}[dr]^{\beta'_3} & \\
V_1 & & V_2 & & V_3 \\
}
$$
Here $U_{1,2},U_{2,3}\supset S$ are neighborhoods and $V_{1,2},V_{2,3}$ are $n$-manifolds.
The maps $\alpha_1,\alpha_2,\alpha'_2,\alpha'_3$ are immersions which are the identity on $S$.
By Lemma~\ref{l:diffeomorphism-to-the-image-locally-closed} they can assumed to be diffeomorphisms onto the image
if $U_{1,2}$ and $U_{2,3}$ are sufficiently close to $S$.
The maps $\beta_1,\beta_2,\beta'_2,\beta'_3$ are arbitrary immersions.

Take the fiber product $U_{1,2,3}=U_{1,2}\times_{U_2}U_{2,3}$.
The differentials of $\alpha_2$ and $\alpha'_2$ are isomorphisms,
so $U_{1,2,3}$ is a manifold
and the maps $\alpha_{1,2}:U_{1,2,3}\to U_{1,2}$ and $\alpha_{2,3}:U_{1,2,3}\to U_{2,3}$ are immersions.
Since $\alpha_2$ and $\alpha'_2$ are injective, the same is true for $\alpha_{1,2}$ and $\alpha_{2,3}$.
Therefore we can assume $U_{1,2,3}$ imbedded into $M$.
Then $\alpha_{1,2}$ and $\alpha_{2,3}$ are the identity on $S$.

Similarly, we take the fiber product $V_{1,2,3}=V_{1,2}\times_{V_2}V_{2,3}$
(it is a manifold since the differentials of $\beta_2$ and $\beta'_2$ are isomorphisms)
and define immersions $\beta_{1,2}:V_{1,2,3}\to V_{1,2}$ and $\beta_{2,3}:V_{1,2,3}\to V_{2,3}$.
We obtain the following commutative diagram:

$$
\xymatrix{
& & & & U_{1,2,3}\ar@{->}[dll]_{\alpha_{1,2}}\ar@{->}[dr]^{\alpha_{2,3}}\ar@{-->}[ddd]_(0.3){\phi_{1,2,3}} & & \\
& & U_{1,2}\ar@{->}[dll]_{\alpha_1}\ar@{->}[dr]^{\alpha_2}\ar@{->}[ddd]_{\phi_{1,2}} &
& & U_{2,3}\ar@{->}[dll]_(0.3){\alpha'_2}\ar@{->}[dr]^{\alpha'_3}\ar@{->}[ddd]_{\phi_{2,3}} & \\
U_1\ar@{->}[ddd]_{\phi_1} & & &
U_2\ar@{->}[ddd]_(0.3){\phi_2} & & &
U_3\ar@{->}[ddd]_{\phi_3} \\
& & & & V_{1,2,3}\ar@{->}[dll]_(0.3){\beta_{1,2}}\ar@{->}[dr]^{\beta_{2,3}} & & \\
& & V_{1,2}\ar@{->}[dll]_{\beta_1}\ar@{->}[dr]^{\beta_2} &
& & V_{2,3}\ar@{->}[dll]_(0.3){\beta'_2}\ar@{->}[dr]^{\beta'_3} & \\
V_1 & & & V_2 & & & V_3 \\
}
$$

By the pullback property of $V_{1,2,3}$
there exists a unique map $\phi_{1,2,3}$ as in the diagram.
Together with the compositions
$\alpha_1\circ\alpha_{1,2}$,
$\alpha'_3\circ\alpha_{2,3}$,
$\beta_1\circ\beta_{1,2}$ and
$\beta'_3\circ\beta_{2,3}$
it establishes a global equivalence of the $S$-germs of $\phi_1$ and $\phi_3$.
\end{proof}

\subsection{Relations between the notions of equivalence}\label{s:equivalence-implications}

Note that the equivalences
have trivial reductions:
globally $L$-equivalent germs are always globally equivalent,
stratumwise equivalent germs are always locally equivalent etc.

To show that the converse implications do not hold in several cases
we will use the following assertion.

\begin{proposition}\label{pr:glued-pairs}
Let $X,Y,Z$ be topological spaces.
Let $g:X\to Y$ and $h:Y\to Z$ be continuous maps and
suppose $h$ is local homeomorphism. 
Take a path-connected subset $W\subset X^2$
such that for every $(x,y)\in W$ we have $h(g(x))=h(g(y))$.
Then if $g(x_0)=g(y_0)$ for some $(x_0,y_0)\in W$, then $g(x)=g(y)$ for every $(x,y)\in W$.
\end{proposition}

\begin{proof}
Suppose there is $(x_1,y_1)\in W$ such that $g(x_1)\ne g(y_1)$.
Since $W$ is path-connected, there is a pair of paths $\gamma_1,\gamma_2:[0;1]\to X$
such that $(\gamma_1(t),\gamma_2(t))$ is a path in $W$ from $(x_0,y_0)$ to $(x_1,y_1)$.

Mark all $s\in[0;1]$ such that $g(\gamma_1(t))=g(\gamma_2(t))$ for every $t\in[0;s]$.
Let $s_0$ be the supremum of such $s$.
By continuity, $g(\gamma_1(s_0))=g(\gamma_2(s_0))$.

On the one hand, $h(g(\gamma_1(t)))=h(g(\gamma_2(t)))$ for all $t\in[0;1]$ by the definition of $W$.
On the other hand, one can choose $t>s_0$ arbitrarily close to $s_0$ such that $g(\gamma_1(t))\ne g(\gamma_2(t))$.
This contradicts the fact that the restriction of $h$ to some neighborhood of $g(\gamma_1(s_0))$ is injective.
\end{proof}

\begin{figure}[h]
\center{\includegraphics{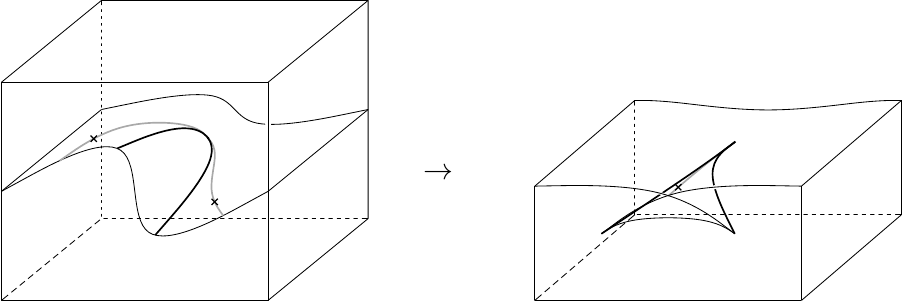}}
\caption{Fold points with the same image near the swallowtail singularity.}\label{fig:swallowtail}
\end{figure}

\begin{remark}\label{r:stratumwise-not-globally}
For germs of stratified immersions at their skeleta,
stratumwise $L$-equivalent germs (and hence stratumwise equivalent germs)
need not be globally equivalent (and moreover, need not be globally $L$-equivalent).

For example, take a standard map $\phi:\R^3\to\R^3$ with a swallowtail singularity $\Sigma^{1,1,1,0}$.
Then the germ of $\phi$ at $S=\Sigma(\phi)$ considered up to global equivalence
has an invariant: the curve of double points of $S$,
see Fig.~\ref{fig:swallowtail}.

Let us construct a map $\phi':\R^3\to\R^3$ which is stratumwise $L$-equivalent to $\phi$.
Take a pair of points $x,y\in S$ with the same image as in Fig.~\ref{fig:swallowtail}.
Take a small ball $\phi(x)\in V\subset \R^3$ and
a diffeomorphism $\psi:V\to V$ which is identity on $\Op(\partial V)$, but not identity at $\phi(x)$.
Let $x\in U_x\subset\R^3$ be a neighborhood including exactly one of the components of $\phi^{-1}(V)$.
Then we set $\phi'=\phi$ outside $U_x\cap\phi^{-1}(V)$ and $\phi'=\psi\circ\phi$ on $U_x\cap\phi^{-1}(V)$.

Proposition~\ref{pr:glued-pairs} implies that the $S$-germs of $\phi$ and $\phi'$ are not globally equivalent
since $\phi(x)=\phi(y)$ and $\phi'(x)\ne\phi'(y)$.
On the other hand, they are stratumwise $L$-equivalent
(indeed, their $\Sigma^{1,0}$-germs are globally $L$-equivalent
to normal maps, which have no double points at the critical locus).

\end{remark}

\begin{remark}\label{r:equivalent-not-l-equivalent}
Globally equivalent germs
(and, so, stratumwise equivalent and locally equivalent)
may not even be locally $L$-equivalent
(and, moreover, may be not stratumwise $L$-equivalent and not globally $L$-equivalent),
because the $L$-equivalence ``remembers'' pairs of regular points near $S$ with the same image.
For example, the germs of maps $\R^2\to\R^2$, defined as $(x,y)\mapsto(x^2,y)$ and $(x,y)\mapsto(x^2,x+y)$,
at their critical locus $x=0$
are globally equivalent but even not locally $L$-equivalent.
%
\end{remark}

\begin{remark}\label{r:locally-not-stratumwise}
In general, locally equivalent germs may not be stratumwise equivalent.
For example, suppose $U=V_1=S^1\times[-1;1]$ and $V_2$ is the M\"{o}bius band.
We set the maps $\phi_1:U\to V_1$ and $\phi_2:U\to V_2$ to be the compositions of the projection to the middle line $U\to S^1$
with the imbeddings of $S^1$ as the middle line into $V_1$ and $V_2$, respectively.
Let $S=S^1\times0$ have only one stratum.
Note that $\phi_1|_S$ and $\phi_2|_S$ are immersions.
Then the $S$-germs of $\phi_1$ and $\phi_2$ are locally $L$-equivalent but not stratumwise equivalent.
\end{remark}


\begin{proposition}\label{pr:if-locally-then-stratumwise}
Let $U$, $V$ and $V'$ be $n$-manifolds 
and let $\phi:U\to V$ and $\phi':U\to V'$ be maps.
Suppose that $\phi$ is a stratified immersion with the skeleton $S$
and that the $S$-germs of $\phi$ and $\phi'$ are locally $L$-equivalent.
Then these germs 
are stratumwise $L$-equivalent.
\end{proposition}

\begin{proof}
Take a stratum $C\subset S$.
We need to prove that $C$-germs of $\phi$ and $\phi'$ are globally $L$-equivalent.
We will use a few results of \S\ref{s:local}.
By Lemma~\ref{l:generic-normal-form} the $C$-germ of $\phi$ has a normal form $\phi''$.
Then by Theorem~\ref{th:if-locally-then-at-stratum} the $C$-germs of $\phi''$ and $\phi'$ are globally $L$-equivalent.
It remains to use the transitivity of globally $L$-equivalence of $C$-germs
(Proposition~\ref{pr:equivalences})
or the universal property of a normal form
(Proposition~\ref{pr:normal-form-universal-property}).
\end{proof}

We have shown that there are the following implications between the equivalences
(implications which may violate in general
but always hold for germs of stratified immersions at their skeleta
are denoted by thin arrows):

%
$$
\xymatrix{
  \text{global equivalence}\ \ar@{=>}[d]    \ar@{<=}[r]
  & \ \text{global }L\text{-equivalence} \ar@{=>}[d] \\
  \text{stratumwise equivalence}\ \ar@{=>}[d] \ar@{<-}[dr] \ar@{<=}[r]
  & \ \text{stratumwise }L\text{-equivalence} \ar@{=>}[d] \\
  \ar@/_-1pc/@{..>}[u]^{???}\text{local equivalence}   \ar@{<=}[r]
  & \ \text{local }L\text{-equivalence} \ar@/_1pc/@{->}[u]_{\text{Proposition}~\ref{pr:if-locally-then-stratumwise}}
}
$$

In this diagram every arrow that goes left or down correspond to trivial reductions.
A~missing up or right arrow means that there is no implication,
see remarks
above.

\begin{question}
Does the stratumwise equivalence of germs of stratified immersions at their skeleta follow from the local equivalence?
\end{question}

We do not know the answer to this question,
so we used only the $L$-versions of the local and stratumwise equivalences.


\end{document}